\documentclass[draft,12pt]{article}

\usepackage{amsfonts,amsmath,amsthm,amscd,amssymb,latexsym,cite,verbatim,texdraw,floatflt,caption2,pb-diagram}

   \usepackage[T2A]{fontenc} 
    \usepackage[cp1251]{inputenc}
    \usepackage[ukrainian]{babel}

\usepackage[mathscr]{eucal}

\newtheorem{theorem}{Теорема}[section]
\newtheorem{lemma}{Лема}[section]
\newtheorem{remark}{Зауваження}[section]
\newtheorem{corollary}{Наслідок}[section]
\newtheorem{proposition}{Твердження}[section]
\theoremstyle{definition}

\renewcommand{\abstract}{\textbf{Анотація. }\medskip}

\numberwithin{equation}{section}

\begin{document}

\textbf{Актуальні проблеми теорії наближень в метриках дискретних просторів на множинах сумовних періодичних та майже періодичних функцій\footnote{\noindent This work was partially supported by a grant from the Simons Foundation (1290607, A.\,Sh.).}}

\vskip 5mm

\textbf{Анатолій~Сердюк$^{1}$, Андрій~Шидліч$^{1,2}$}

\vskip 5mm

$^{1}$Інститут математики Національної академії наук України, м. Київ

$^{2}$Національний університет біоресурсів і природокористування України

\vskip 7mm

\textbf{Анотація.} У даній оглядовій роботі висвітлюються основні аспекти розвитку досліджень, пов'язаних з розв'язанням екстремальних задач теорії апроксимації у просторах ${\mathcal S}^p$ та $B{\mathcal S}^p$  відповідно періодичних  та  майже періодичних сумовних функцій, у яких $l_p$-норми  послідовностей коефіцієнтів Фур'є є скінченними. Зокрема, огляд мiстить відомі на цей час результати, що стосуються найкращих, найкращих $n$-членних наближень та поперечникiв класів функцій однієї та багатьох змінних, які означаються за допомогою $\psi$-похідних та узагальнених модулів гладкості в просторах ${\mathcal S}^p$ та $B{\mathcal S}^p$. Особливу увагу приділено розвитку досліджень, пов'язаних з отриманням прямих та обернених апроксимаційних теорем у цих просторах.

\vskip 10mm

\textbf{Actual problems of the approximation theory in metrics of discrete spaces on sets of summable periodic and almost periodic functions}

\vskip 5mm

\textbf{Anatolii Serdyuk$^{1}$, Andrii Shidlich$^{1,2}$}

\vskip 5mm

$^{1}$Institute of Mathematics of the National Academy of Sciences of Ukraine, Kyiv

$^{2}$National University of Life and Environmental Sciences of Ukraine, Kyiv

\vskip 7mm

\textbf{Abstract} This review paper highlights the main aspects of the development of research related to the solution of extreme problems in the theory of approximation in the spaces ${\mathcal S}^p$ and $B{\mathcal S}^p$ of periodic and almost periodic summable functions, respectively, where the $l_p$-norms of the sequences of Fourier coefficients are finite. In particular, the review contains the results known so far concerning the best, best $n$-term approximations and widths of classes of functions of one and many variables defined by means of $\psi$-derivatives and generalized moduli of smoothness in the spaces ${\mathcal S}^p$ and $B{\mathcal S}^p$. Particular attention is paid to the development of studies related to the derivation of direct and inverse approximation theorems in these spaces.

\vskip 7mm

\textbf{Emails:} sanatolii@ukr.net, shidlich@gmail.com

\newpage
\section{Вступ}\label{Definitions_Sp}

В роботі висвітлюються основні аспекти розвитку досліджень, пов'язаних з розв'язанням низки важливих екстремальних задач теорії апроксимації у просторах
${\mathcal S}^p$ та $B{\mathcal S}^p$  сумовних функцій, у яких $l_p$-норми  послідовностей коефіцієнтів Фур'є є скінченними. Простори такого виду відомі досить давно.
Зокрема, при $p=1$ вони є просторами абсолютно сумовних функцій (див., наприклад, \cite{Bari_1961M, Kahan_M1976}), а при $p=2$ -- гільбертовими просторами. В Інституті математики НАН України систематичні дослідження, пов'язані зі  знаходження точних розв'язків класичних екстремальних задач теорії апроксимації у таких просторах почали проводитись, починаючи з 2000-х років з ініціативи О.\,І.~Степанця і за активної підтримки його наукової школи.
Інтерес до цих досліджень виник в результаті  пошуку нових підходів до задач теорії наближення функцій багатьох змінних, де, на відміну від одновимірного випадку, дуже рідко вдавалось отримати точні або асимптотично точні результати. Серед визначальних проблем в цій теорії, напевно, варто виділити  такі:
вибір апроксимативних агрегатів, класів функцій та відповідних апроксимаційних характеристик. В той час, як в
одновимірному випадку структура найпростішого агрегату наближення визначається природним порядком натурального ряду, в багатовимірному  випадку, тобто, коли задано простір $L=L_1({\mathbb R}^d)$, $d=1,2,\ldots$ $(d\in {\mathbb N})$, сумовних за Лебегом функцій $f({\bf x})=f(x_{1} \ldots,x_{d})$  $d$ дійсних змінних    вибір найпростіших агрегатів є дещо проблематичним.
Перші труднощі тут починаються з того, що саме слід вважати  аналогом частинної суми для кратного ряду
  \begin{equation}\label{B.1}
  \sum_{{\bf k}\in {\mathbb Z}^d} c_{\bf k},\quad {\bf k}=(k_{1}, \ldots k_{d}),
  \end{equation}
 де ${\mathbb Z}^d$~--- цілочисельна решітка в ${\mathbb R}^{d}$. Природним є розгляд  ``прямокутних''\ сум
 і відповідних їм апроксимативних агрегатів -- у періодичному випадку тригонометричних поліномів вигляду
 \begin{equation}\label{B.2}
 \sum_{k_1=-n_1}^{n_1} \cdots \sum_{k_d=-n_d}^{n_d} c_{k_{1}, \ldots, k_{d}} {\mathrm e}^{{\mathrm i}(k_{1}t_{1}+\cdots+k_{d}t_{d})}.
  \end{equation}
Проте  частинні суми кратного ряду можна означати багатьма іншими способами, зокрема, у такий спосіб.
Нехай $\{G_{\alpha}\}$~-- сім'я  обмежених областей, які залежать
від параметра $\alpha$, $\alpha\in {\mathbb N}$, і такі, що будь-який вектор ${\bf n}\in {\mathbb Z}^d$
належить усім областям $G_{\alpha}$ при достатньо великих значеннях $\alpha$.
 Тоді вираз
 $
 \sum_{{\bf k}\in G_{\alpha}} c_{\bf k}
 $ 
називають частинною сумою ряду (\ref{B.1}), яка відповідає  області $G_{\alpha}$. За аналогією з цим вводяться і відповідні частинні суми тригонометричних рядів:
 \begin{equation}\label{B.3}
 \sum_{{\bf k}\in G_{\alpha}}c_{\bf k}{\mathrm e}^{{\mathrm i}{\bf k} {\bf x}} = \sum_{{\bf k}\in G_{\alpha}} c_{k_1 \ldots k_d} {\mathrm e}^{{\mathrm i}(k_1x_1+\cdots+k_d x_d)}.
 \end{equation}

Досить швидко виявилось, що у випадку наближення функцій з відомих класів Соболєва $W_p^r({\mathbb R}^d)$ замість прямокутних
сум вигляду (\ref{B.2}) доцільніше  застосовувати суми (\ref{B.3}), які побудовані за областями, що
визначаються деякими   гіперболічними поверхнями.
Такі області вперше були введені  К.\,І.~Бабенком в \cite{Babenko_1960_2, Babenko_1960_5} і отримали назву
гіперболічних хрестів. Їх поява дала істотний поштовх у розвитку сучасної теорії наближення функцій багатьох змінних.
В цьому напрямку отримано велику кількість важливих та цікавих результатів, з яким можна ознайомитись, наприклад, з робіт
\cite{Temlyakov_B1993, Romanyuk_2012, Dung_Temlyakov_Ullrich_2018}.

 Слід зазначити, що більшість результатів, які стосуються наближення функцій з використанням
 гіперболічних хрестів, у просторах $L_p({\mathbb R}^d)$ мають порядковий характер, а точні рівності  отримуються лише у
 гільбертових просторах (при $p=2$). Спроби використання гіперболічних хрестів, а також їх модифікацій -- ступінчастих гіперболічних хрестів
 при наближенні функцій з класів, відмінних від соболєвських,
 взагалі кажучи, бажаних результатів майже не дають.  У зв'язку з цим при знаходженні точних розв'язків важливих екстремальних задач  апроксимації
 конкретного функціонального класу $\mathfrak N$ (або ж деякої сім'ї таких класів)
 природно підбирати відповідну йому сім'ю областей $G_{\alpha}$, яка визначається його параметрами.
 Так,  у 2000 році  О.\,І.~Степанцю \cite{Stepanets_Preprint_2000, Stepanets_UMZh2001_3}
 (див. також далі п.~\ref{WidthSubsection})  для множин $\psi$-інтегралів функцій з одиничних куль 
 просторів ${\mathcal S}^p$ вдалося підібрати відповідне сімейство областей і показати, що
підпростір поліномів, побудованих за такими областями є підпростором, на якому реалізуються точні значення колмогоровських поперечників цих множин у заданих просторах  (див. теореми \ref{Applic_Psi_q<p} та \ref{Applic_Kolmogorov_W}).


Іншою причиною, яка ускладнює отримання точних результатів  по наближенню функцій багатьох змінних
є історично сформована практика постановки задач апроксимації саме у просторах Лебега $L_p({\mathbb R}^d)$.
У періодичному випадку норма в цих просторах  означається рівністю
 \begin{equation}\label{Lp_norm}
 \|  f\|_{_{\scriptstyle L_p ({\mathbb  T}^d)}} =
 \bigg((2\pi)^{-d}\int
 _{{\mathbb T}^d} |f({\bf x})|^p {\mathrm d}{\bf x}\bigg)^\frac{1}{p},\quad  {\mathbb T}^d=[0,2\pi)^d.
 \end{equation}
Проте при    $p=2$  простори $L_p({\mathbb R}^d)=L_2({\mathbb R}^d)$ є гільбертовими і крім рівності \eqref{Lp_norm} норму функції $f$ можна також визначити з урахуванням рівності Парсеваля в інший спосіб:
 \begin{equation}\label{L2_norm}
 \| f\|_{_{\scriptstyle L_2 ({\mathbb  T}^d)}} = \bigg(\sum\limits_{{\bf k}\in {\mathbb Z}^d } |\widehat{f}({\bf k})|^2 \bigg)^\frac{1}{2},
 \end{equation}
де  $\widehat{f}({\bf k})$ -- коефіцієнти Фур'є функції $f$. Тому узагальнення гільбертових просторів на випадок інших $p$, відмінних від 2, можна також проводити двома способами. У першому випадку, норма визначається рівністю \eqref{Lp_norm} і характеризує величину середнього значення $p$-го  степеня модуля заданої функції. У другому -- розглядається наступна величина,  яка  повністю характеризується множиною $\{ \widehat{f}({\bf k})\}_{{\bf k}\in   {\mathbb Z}^d}$:
 \begin{equation}\label{Sp_norm_intro}
 \|  f\|_{_{\scriptstyle {\mathcal S}^p ({\mathbb  T}^d)}} =\|\{ \widehat{f}({\bf k})\}_{{\bf k}\in   {\mathbb Z}^d}\|_{l_p({\mathbb Z}^d)}=
 \bigg(\sum\limits_{{\bf k}\in {\mathbb Z}^d } |\widehat{f}({\bf k})|^p \bigg)^\frac{1}{p}.
 \end{equation}
В даному огляді наведено результати, які стосуються другого із зазначених способів. Такий підхід був реалізовний, зокрема, в роботах \cite{Sterlin_1972, Kahan_M1976, Dyachenko_2005, Dyachenko_2007} та ін. Істотний просув у цьому напрямку, на нашу думку, вдалося зробити завдяки дослідженням  О.\,І.~Степанця та послідовників
(див., наприклад, \cite{Stepanets_Preprint_2000, Stepanets_UMZh2001_3, Stepanets_Serdyuk_UMZh2002, Stepanets_Rukasov_UMZh2003_2, Stepanets_UMZh2003_10, Serdyuk_2003, Voicexivskij_2003, Voicexivskij_UMZh2003, Vakarchuk_2004, Vakarchuk_Shchitov_2006, Stepanets_UMZh2006_1, Stepanets_M2005,  Savchuk_Shidlich_2014, Timan_M2009, Shidlich_2016, Abdullayev_Ozkartepe_Savchuk_Shidlich_2019} та ін.).
Зокрема, окрім розв'язку згаданої вище задачі про колмогоровські поперечники множин $\psi$-інтегралів у просторах ${\mathcal S}^p$, О.\,І.~Степанцю та його учням вдалося знайти точні значення найкращих $n$-членних наближень,  тригонометричних поперечників, наближень лінійними методами  таких множин, довести прямі та обернені апроксимаційні теореми та вказати точні константи у відповідних нерівностях. При цьому результати, отримані для просторів  ${\mathcal S}^p$, у багатьох випадках виявилися новими навіть для гільбертових просторів. У роботах \cite{Serdyuk_Shidlich_2022} та \cite{Chaichenko_Shidlich_Shulyk_2022}  (див. також п. \ref{BSp_section}) згадані вище ідеї та підходи О.\,І.~Степанця,  були розповсюджені на множини майже періодичних функцій в сенсі Безиковича, в результаті чого виникли простори Безтковича-Степанця $B{\mathcal S}^p$ та Безиковича-Мусєлака-Орлича $B{\mathcal S}_M$. Для цих просторів отримано прямі та обернені апроксимаційні теореми (див. п. \ref{BSp_section}.4 та \ref{BSp_section}.5), які без втрати точності охоплюють як частинний випадок простори періодичних функцій ${\mathcal S}^p$.

Даний підхід також дозволив розповсюджувати ідеї та методи теорії наближень  на абстрактні лінійні простори  і отримувати завершені  змістовні результати.
У  роботі авторів \cite{Serdyuk_Shidlich_2021} наведено огляд результатів, пов'язаних із екстремальними  задачами теорії наближень в абстрактих лінійних просторах, у яких норми визначаються аналогічно до рівності \eqref{Sp_norm_intro}. Зазначений огляд  містив також деякі результати у просторах  ${\mathcal S}^p$, проте в ньому залишилась неохопленою низка результатів, які, зокрема, стосуються прямих та обернених теореми у просторах майже періодичних функцій, поперечників множин, що визначаються усередненими модулями гладкості та ін.

\section{Основні означення}\label{Definitionss}

\subsection{Означення просторів ${\mathcal S}^p$}\label{Definitions_Sp}

Нехай ${\mathbb R}^d$ --  $d$-вимірний, $d\in {\mathbb N}$, евклідів простір, ${\bf x}=(x_1,\ldots ,x_d)$ -- його елементи,
${\mathbb Z}^d$ -- цілочисельна гратка в ${\mathbb R}^d$, тобто, множина векторів ${\bf k}=(k_1,\ldots ,k_d)$ з
цілочисельними координатами, $({\bf x},{\bf y})=x_1y_1+\cdots
+x_dy_d$ і $|{\bf x}|=\sqrt {({\bf x},{\bf x})}$.

Нехай, далі,  $L=L({\mathbb T}^d)$  -- простір всіх $2\pi$-періодичних за кожною зі
 змінних комплекснозначних функцій $f({\bf x})=f(x_1,\cdots ,x_d)$, сумовних на   ${\mathbb T}^d:=[0,2\pi)^d$. Через ${\mathcal S}^p={\mathcal S}^p({\mathbb T}^d)$, $p\in (0,\infty]$, позначають
простір всіх функцій $f\in L$ зі скінченною  \mbox{(квазі-)нормою}
  \begin{equation}\label{w1}
   \|f\|_{_{\scriptstyle {\mathcal S}^p}}:=\|f\|_{_{\scriptstyle {\mathcal S}^p({\mathbb T}^d)}}= \|\widehat f({\bf  k})\|_{l_p({\mathbb Z}^d)}=\left\{\begin{matrix} \Big(\sum _{{\bf k}\in {\mathbb Z}^d}|\widehat f({\bf  k})|^p\Big)^{1/p},\quad\hfill & p\in (0,\infty),\\
   \sup_{{\bf k}\in {\mathbb Z}^d}|\widehat f({\bf  k})|,\quad\hfill & p=\infty,
   \end{matrix}\right.
 \end{equation}
де
\begin{equation}\label{Fourier_Coeff}
\widehat f({\bf k}):=(2\pi)^{-d}\int_{\mathbb T^d}f({\bf x}){\mathrm e}^{-{\mathrm i}({\bf k},{\bf x})}
 {\mathrm d}{\bf x},\quad{\bf k}\in\mathbb Z^d,
    \end{equation}
 -- коефіцієнти Фур'є функції    $f$ за   системою  $\{{\mathrm e}^{{\mathrm i}({\bf k},{\bf x})}\}_{{\bf k}\in {\mathbb Z}^d}$.

Як зазначено вище, при $p=2$ внаслідок \eqref{L2_norm} простори ${\mathcal S}^p$ є гільбертовими просторами $L_2({\mathbb T}^d)$. При інших $p>0$ ці простори також наслідують деякі важливі властивості, притаманні гільбертовим просторам, зокрема, мінімальну властивість сум Фур'є, яка формулюється в такий спосіб.

\begin{proposition}{\rm (\cite[Гл.~11]{Stepanets_M2005}, \cite{Stepanets_UMZh2006_1})}\label{FS_P}
 Нехай  $f$ -- довільна функція з простору ${\mathcal S}^p$, $ p\in (0, \infty ]$, і
\begin{equation}\label{Fourier_Seriesss}
  S[f]({\bf x})=\sum_{{\bf k}\in {\mathbb Z}^d}\widehat f({\bf k}){\mathrm e}^{{\mathrm i}({\bf k},{\bf x})}
    \end{equation}
 її ряд Фур'є за системою  $\{{\mathrm e}^{{\mathrm i}({\bf k},{\bf x})}\}_{{\bf k}\in {\mathbb Z}^d}$.  Нехай, далі, $\{G_{\alpha}\}$~-- сім'я  обмежених областей, які залежать від параметра $\alpha$, $\alpha\in {\mathbb N}$, і такі, що будь-який вектор ${\bf n}\in {\mathbb Z}^d$ належить усім областям $G_{\alpha}$ при достатньо великих значеннях $\alpha$.

Для фіксованої  області $G_\alpha$ серед усіх сум вигляду
 \begin{equation}\label{G_alpha}
  \sum_{{\bf k}\in G_\alpha} c_{\bf k}
  {\mathrm e}^{{\mathrm i}({\bf k},{\bf x})},
  \quad c_{\bf k}\in {\mathbb C},
 \end{equation}
в просторі ${\mathcal S}^p$ найменше відхиляється  від $ f$ частинна  сума $S_{_{\scriptstyle G_\alpha}}(f)({\bf x})=\sum_{{\bf k}\in G_\alpha}\widehat f({\bf k}){\mathrm e}^{{\mathrm i}({\bf k},{\bf x})}$ ряду \eqref{Fourier_Seriesss},  тобто, суми $S_{_{\scriptstyle G_\alpha}}(f)$ реалізують найкращі наближення функції $f$ тригонометричними поліномами вигляду \eqref{G_alpha} в просторі ${\mathcal S}^p$:
$$
 E_{_{\scriptstyle G_\alpha}}(f)_{_{\scriptstyle {\mathcal S}^p}}:=\inf\limits_{c_{\bf k}\in {\mathbb C}}\Big\|f- \sum_{{\bf k}\in G_\alpha} c_{\bf k}{\mathrm e}^{{\mathrm i}({\bf k},\cdot)}\Big\|_{_{\scriptstyle p}} = \|f - S_{_{\scriptstyle G_\alpha}}(f)\|_{_{\scriptstyle p}}.
$$
 \end{proposition}

\subsection{ Поняття $\psi$-інтеграла, $\psi$-похідної та характеристичних послідовностей }\label{psi-derivatives}

\noindent{\bf \ref{psi-derivatives}.1. } Для означення класів наближення в просторах  ${\mathcal S}^p$ використовуються поняття $\psi$-інтеграла та $\psi$-похідної  \cite{Stepanets_UMZh2001_3}, \cite[Гл.~11]{Stepanets_M2005}. Нехай $ {\psi }=\{\psi ({\bf k})\}_{{\bf k}\in {\mathbb Z}^d}$ -- довільна система комплексних чисел -- кратна
послідовність. Якщо для  функції $f\in L$  з рядом Фур'є
\begin{equation}\label{b8}
    S[f]({\bf x})=\sum\limits_{{\bf k}\in {\mathbb Z}^d}\widehat{f} ({\bf k}){\mathrm e}^{{\mathrm i}({\bf k},
    {\bf x})}
    \end{equation}
ряд
  \begin{equation}\label{w2}
\sum _{{\bf k}\in {\mathbb Z}^d}\psi ({\bf k})\widehat {f} ({\bf
k}){\mathrm e}^{{\mathrm i}({\bf k},{\bf x})}. 
 \end{equation}
є рядом Фур'є  деякої функції  $F$ з  $L,$ то  $F$ називають $\psi$-інтегралом функції  $f$ і позначають
$F={\mathcal J}^{\psi } (f)$.  При цьому функцію  $f$ називають $\psi $-похідною функції  $F$ і позначають  через $f=D^{\psi }(F)=F^{\psi }$.
Множину $\psi$-інтегралів всіх функцій   $f\in L$ позначають через $L^{\psi}.$ Якщо  ${\mathfrak {N}}$ -- деяка підмножина з $L,$ то
через $L^{\psi }{\mathfrak {N}}$ позначають  множину  $\psi$-інтегралів всіх функцій з ${\mathfrak {N}}.$
Зрозуміло, що коли $f\in L^{\psi },$  коефіцієнти Фур'є функцій $f$ и $f^{\psi }$ пов'язані співвідношеннями
 \begin{equation}\label{w3}
  \widehat f({\bf k})=\psi ({\bf k})\widehat f^{\psi }({\bf k}), \ \ {\bf k}\in {\mathbb Z}^d. 
 \end{equation}
В ролі ${\mathfrak {N}}$ можна обрати одиничну кулю в  $U^p$ в просторі ${\mathcal S}^p $:
\begin{equation}\label{w4}
 U^p=\{f\in {\mathcal S}^p: \quad  \|f\|_p\le 1\}. 
 \end{equation}
В такому випадку покладаємо  $L_p^{\psi }:=L^{\psi }U^p.$

Зазначимо, що коли
  \begin{equation}\label{b15}
  \psi_k \neq 0 \quad \forall k \in {\mathbb N},
  \end{equation}
то внаслідок (\ref{w3}) та (\ref{w4})
  \begin{equation}\label{b16}
      L_p^{\psi } = \bigg\{f \in L:\quad \sum_{{\bf k}\in {\mathbb Z}^d} \bigg |\frac{\widehat f({\bf k})}{\psi({\bf k})}\bigg|^{p}
  \leq 1  \bigg\},\quad p\in (0,\infty),
    \end{equation}
тобто, множина $L_p^{\psi }$ складається із функцій простору $L$, послідовності коефіцієнтів Фур'є яких належать $p$-еліпсоїду простору $l_p$ з півосями, які дорівнюють $|\psi({\bf k})|$.

Надалі обмежимося випадком, коли система $\psi$ задовольняє умову
\begin{equation}\label{w5}
 \lim _{|{\bf k}|\to \infty }\psi ({\bf k})=0. 
  \end{equation}
Якщо $f\in L^{\psi }{\mathcal S}^p$ і  $|\psi ({\bf k})|\le K,$ ${\bf k}\in {\mathbb Z}^d,$  то
$f\in {\mathcal S}^p$. Тому за  умови (\ref{w5}) має місце  вкладення   $L_p^{\psi }\subset {\mathcal S}^p.$


\vskip 2mm
\noindent{\bf \ref{psi-derivatives}.2.}
 Конструкцію поліномів, які використовуються
 для наближення елементів $f \in {\mathcal S}^p$, зручно визначати за допомогою спеціально підібраних
 характеристичних послідовностей $\varepsilon(\psi),\,g(\psi)$ і $\delta(\psi)$ системи $\psi$, які задаются в такий спосіб \cite{Stepanets_UMZh2001_3}, \cite[Гл.~11]{Stepanets_M2005}.

Нехай $ {\psi }=\{\psi ({\bf k})\}_{{\bf k}\in {\mathbb Z}^d}$ -- довільна система комплексних чисел, які задовольняють умову
(\ref{w5}). Через $\varepsilon(\psi )=\varepsilon_1, \varepsilon_2, \ldots $ позначимо множину значень величин $|\psi ({\bf k})|,$ ${\bf k}\in
{\mathbb Z}^d,$ впорядковану за спаданням, через $g(\psi) =\{ g_1,g_2,\ldots\} $ -- послідовність
множин
   \begin{equation}\label{b17iq}
  g_n=g_n(\psi)=\left\{{\bf k}\in {\mathbb Z}^d:  |\psi ({\bf k})|\ge \varepsilon_n
 \right\}
  \end{equation}
 і через $\delta(\psi) = \delta_1,\delta_2,\ldots $ --  послідовність  чисел $\delta_n =|g_n|,$ де $|g_n|$ --
кількість чисел ${\bf k}\in {\mathbb Z}^d$, які належать  множині $g_n$.  Через  $g_0 =g_0(\psi)$ позначають порожню
множину і вважають, що $\delta_{0}=0$.

 Враховуючи умову (\ref{w5}), послідовності $\varepsilon(\psi)$ і $g(\psi)$ можна визначити
 рекурентно за допомогою співвідношень:
   \begin{equation}\label{b17i}
\displaystyle{\begin{matrix}\varepsilon_1=\sup\limits_{{\bf k}\in {\mathbb
Z}^d}|\psi ({\bf k})|,\   g_1=\{{\bf k}\in {\mathbb Z}^d:\
|\psi ({\bf k})|=\varepsilon_1\},\quad \varepsilon_n=\sup\limits_{{\bf k}\in {\mathbb Z}^d\setminus g_{n-1}}|\psi ({\bf k})|,\cr\\ \
g_n=g_{n-1}\cup \{{\bf k}\in {\mathbb Z}^d: \ |\psi ({\bf k})|=\varepsilon_n\},\quad n\in {\mathbb N}\setminus\{1\}.
\end{matrix}}
  \end{equation}
За такого означення будь-яке число $n^{\ast} \in {\mathbb N}$ належить усім
множинам $g_n $ з достатньо великими номерами $n$ і $
  \lim\limits_{k\rightarrow\infty}\,\delta_k = \infty.
$

Зазначимо також, що якщо $\widetilde{\psi}=\{\widetilde{\psi}_k\}_{k=1}^\infty$ -- спадна
перестановка системи чисел $|\psi ({\bf k})|$, $k=1,2,\ldots$, то має місце рівність
     \begin{equation}\label{b18}
     \widetilde{\psi}_k=\varepsilon_n\quad \forall k\in
     (\delta_{n-1},\delta_n],\ \ n=1,2,\ldots .
     \end{equation}

Окрім природної умови  (\ref{w5}) від систем $\psi $ жодних інших істотних обмежень не вимагатється.
Тому ці системи  $\psi $, а з ними і їх характеристичні послідовності $\varepsilon(\psi ),$
та $\delta (\psi )$ в загальному випадку можуть бути різноманітними, а сукупності областей $g(\psi )$ -- мати доволі складну конфігурацію.

\vskip 2mm
\noindent{\bf \ref{psi-derivatives}.3. } Дослідження зв'язку між гладкісними властивостями неперервної періодичної функції  та  швидкістю спадання до нуля її найкращих рівномірних наближень тригонометричними поліномами почалися ще на початку минулого століття у роботах  А.~Лебега, Д.~Джексона, С.\,Н.~Бернштейна, Ж.~де ла Валле Пуссена (див., наприклад, роботу \cite{Bernstein_1912} та посилання у ній). Ідея класифікації періодичних функцій на основі перетворень їхніх рядів Фур'є за допомогою мультиплікаторів  виникла під впливом досліджень робіт А.~Вейля (див., наприклад, \cite{Akhiezer_M1947}), Б.~Надя \cite{Sz.-Nagy_1948}, П.~Бутцера та Р.~Несселя \cite{Butzer_Nessel_1971M} та інших. В завершеному вигляді така класифікація була реалізована у 80-х роках минулого століття О.\,І.~Степанцем (див., наприклад, монографії \cite{Stepanets_M1987, Stepanets_M1995, Stepanets_M2002_1, Stepanets_M2002_2, Stepanets_M2005}).

Зазначимо, що спочатку у 1983 році для функцій однієї дійсної змінної О.\,І.~Степанцем запроваджено поняття  $(\psi,\beta)$-похідної \cite{Stepanets_Preprint_1983}. У цьому випадку  $\psi=\{\psi(k)\}_{k\in {\mathbb Z}}$ -- система відмінних від нуля комплексних чисел таких, що
 \[
 \psi(k)=|\psi(k)|\cos \frac{\beta\pi}2 - {\mathrm i}|\psi(k)|\sin \frac{\beta\pi}2,\quad k=0,1,\ldots, \quad \beta\in {\mathbb  R},
 \]
а значення $\psi(k)$ та ${\psi(-k)}$ є комплексно спряженими, тобто $\psi(-k)=\overline {\psi(k)}$, $k\in {\mathbb Z}$. Згодом з'явились узагальнення цього поняття, а також оберненого до нього поняття $\psi$-інтеграла як для періодичних функцій однієї та багатьох змінних, так і для довільних функцій, заданих на всій дійсній осі. Разом з поняттям $\psi$-похідної  О.\,І.~Степанець ввів у розгляд допоміжні характеристики -- модулі напіврозпаду, які дозволили ранжувати функції $\psi$, а разом з ними і відповідні множини    $(\psi,\beta)$-диференційовних  функцій, в залежності від швидкості спадання до нуля.
Такий підхід дав змогу здійснити класифікацію та розробити ефективні методи дослідження класичних задач теорії апроксимації для множини усіх сумовних  функцій, починаючи від функцій малої гладкості, ряди Фур'є яких можуть розбігатися,  і закінчуючи нескінченно диференційовними функціями,  аналітичними та цілими функціями \cite{Stepanets_M2002_1, Stepanets_M2002_2, Stepanets_M2005}.

Питання класифікації множини  нескінченно диференційовних, аналітичних та цілих функцій в термінах $(\psi,\beta)$-похідних досліджувались, зокрема, в \cite{Stepanets_Serdyuk_Shidlich_2007, Stepanets_Serdyuk_Shidlich_2008, Stepanets_Serdyuk_Shidlich_2009}.


\vskip 2mm
\noindent{\bf \ref{psi-derivatives}.4. } В багатовимірному випадку  природними є
системи $\psi,$ у яких    $\psi({\bf k})$ зображуються у вигляді
\begin{equation}\label{w35}
 \psi ({\bf {k}})=\psi (k_1, \ldots ,k_d)=\prod _{j=1}^d\psi
 _j({k}_j), \ \ {k_j}\in {\mathbb Z}^1, \ \ j=\overline {1,d}, 
  \end{equation}
значень одновимірних послідовностей $\psi _j=\{\psi
_j(k_j)\}_{k_j=1}^{\infty }.$ Якщо при цьому $\psi_j(k_j)$ і ${\psi _j(-k_j)}$ є комплексно спряженими,
то множини $g_n(\psi)$ є симетричними відносно усіх координатних площин і
 \begin{equation}\label{w36}
 \sum _{{\bf k}\in {\mathbb Z}^d}\psi ({\bf k}){\rm e}^{{\rm i}{\bf
 k}{\bf
 t}}=\sum _{{\bf k}\in {\mathbb Z}_{+}^d}2^{d- q({\bf
 k})}\prod _{j=1}^d |\psi _j(k_j)|\cos \Big(k_jt_j - \frac {\beta
 _{k_j}\pi }2\Big), 
   \end{equation}
де  $q({\bf k})$ -- кількість координат вектора ${\bf k}$, які дорівнюють  нулю, а
числа $\beta _{k_j}$ означаються співвідношеннями
 $$
 \cos \frac {\beta _{k_j}\pi }2=\frac {{\rm Re} \ \psi _j(k_j)}{|\psi
 _j(k_j)|}, \ \ \ \sin \frac {\beta _{k_j}\pi }2=\frac {{\rm Im }\
 \psi _j(k_j)}{|\psi _j(k_j)|}.
 $$
 В такому випадку множина $L^{\psi}$ $\psi$-інтегралів дійсних функцій
  $\varphi $ з $L({\mathbb T}^d)$ складається з дійснозначних функцій $f$, і якщо до того ж ряд в
  (\ref{w2}) є рядом Фур'є деякої сумовної  функції ${K}_{\psi },$ то
  достатньою умовою включення  $f\in L^{\psi }{\mathfrak {N}}$
  є можливість зображення функції  $f$ у вигляді згортки
  $$
  f({\bf x})=(2\pi )^{-d}\int\limits _{{\mathbb T}^{d}}\varphi
  ({\bf x}-{\bf t}){K}_{\psi }(t){\mathrm d}{\bf t},
  $$
де $\varphi \in {\mathfrak {N}}$ і майже скрізь  $\varphi ({\bf x})=f^{\psi }({\bf x})$.
Це, зокрема, означає, що множини $L^{\psi}{\mathfrak {N}}$ включають класи функцій, які зображуються у вигляді згорток
з фіксованими сумовними ядрами.

\vskip 2mm
\noindent{\bf \ref{psi-derivatives}.5. } Нехай $l_r^d$, $0<r\le \infty$, -- простір всіх послідовностей ${\bf x}=\{x_k\}_{k=1}^d\in {\mathbb R}^d$ зі стандартною $l_r$-нормою (квазі-нормою)
 \[
 |{\bf x}|_r:=\|{\bf x}\|_{_{\scriptstyle l_r^d({\mathbb N})}}=\left\{\begin{matrix}
 (\sum_{k=1}^d |x_k|^r)^{1/r},\quad 0<r<\infty,\\ \sup_{1\le k\le d}|x_k|,\quad\quad\quad r=\infty.\end{matrix}\right.
 \]

Нехай, далі, $\psi=\psi(t)$, $t\ge 1$, -- деяка  фіксована додатна спадна до нуля
функція, $\psi(0):=\psi(1)$.
Одним із  важливих прикладів множин $L_p^{\psi^*}$ є  класи
${\mathcal F}_{p,r}^{\psi}$, $0<p,r\le \infty$,  які отримуються з останніх у випадку, коли система  $\psi^*({\bf k})=\psi(|{\bf k}|_{r})$, ${\bf k}\in {\mathbb Z}^d$, тобто,
 \[
  {\mathcal F}_{p,r}^{\psi}={\mathcal F}_{p,r}^{\psi}({\mathbb T}^d):=\bigg\{f\in L({\mathbb T}^d):\  \|\{\widehat{f} ({\bf k})/\psi(|{\bf k}|_{r})\}_{{\bf k}\in {\mathbb Z}^d}\|_{l_p({\mathbb Z}^d)}\le 1\bigg\}.
 \]
Зазначимо, що при  $\psi(t)=t^{-s}$, $s\in {\mathbb N}$ і $p=1$,  класи \label{F_qr^s}  ${\mathcal F}_{p,\infty}^{\psi}=:{\mathcal F}_{p,\infty}^{s}$ є множинами функцій, у яких частинні похідні порядку $s$ мають абсолютно збіжні ряди Фур'є. Якщо ж $p=2$, то класи  ${\mathcal F}_{2,\infty}^{s}$ збігаються з класами  Соболєва $W^s_2$. Апроксимативні характеристики класів ${\mathcal F}_{p,r}^{\psi}$  для різних $r\in (0,\infty]$ і різноманітних функцій $\psi$ досліджувались в роботах \cite{DeVore_Temlyakov_1995, Temlyakov_Greedy_1998, Li_2010, Shidlich_Zb2011, Shidlich_Zb2013, Shidlich_Zb2014, Serdyuk_Stepanyuk_2015, Shidlich_2016} та інших. Отримані для цих класів результати знайшли своє застосування до дослідження  апроксимативних характеристик  функціональних класів у просторах Лебега $L_p({\mathbb T}^d)$ (див., зокрема, п.~\ref{Apll_Lp}).

\section{Апроксимативні характеристики функціональних класів у просторах  ${\mathcal S}^p$}


\subsection{Наближення індивідуальних функцій у просторах ${\mathcal S}^p$} Нехай  $\psi =\{\psi ({\bf k})\}_{{\bf k}\in {\mathbb Z}^d} $ -- система чисел, підпорядкована умовам \eqref{b15}  та  \eqref{w5}.

Величину
    \begin{equation}\label{b21}
          E_n(f)_{\psi,\,p} = \inf_{c_k\in {\mathbb C}} \Big\|\,f - \sum_{{\bf k}\in g_n(\psi)}
          c_{\bf k} {\mathrm e}^{{\mathrm i}({\bf k},\cdot)}
          \Big\|_{_{\scriptstyle {\mathcal S}^p}}
    \end{equation}
називають найкращим наближенням функції $f$ в просторі ${\mathcal S}^p$ тригонометричними поліномами, номери гармонік яких належать областям  $g_{n}(\psi)$ вигляду \eqref{b17iq}.

\begin{theorem}{\rm (\cite{Stepanets_UMZh2001_8}, \cite[Гл.~11]{Stepanets_M2005})}\label{Applic_Direct_Th}  Нехай
$ f\in L^{\psi }{\mathcal S}^p$, $
p>0,$ і $ \psi =\{\psi (\bf k)\}_{{\bf k}\in {\mathbb Z}^d} $
-- система чисел, яка задовольняє умову  \eqref{w5}. Тоді ряд
 $$
 \sum _{k=1}^{\infty }(\varepsilon_k^p -
 \varepsilon_{k-1}^p)E_k^p(f^{\psi })_{_{\scriptstyle \psi ,{\mathcal S}^p}}
 $$
 збігається, і при кожному $ n\in {\mathbb N}$ має місце рівність
 $$
 E_n^p(f)_{_{\scriptstyle \psi ,{\mathcal S}^p}}=\varepsilon_n^pE_n^p(f^{\psi })_{\psi
 ,p} + \sum _{k=n+1}^{\infty
 }(\varepsilon_k^p-\varepsilon_{k-1}^p)E_k^p(f^{\psi })_{_{\scriptstyle \psi ,{\mathcal S}^p}},
 $$
 де  $ \varepsilon_k $ -- елементи характеристичної послідовності  $ \varepsilon(\psi )$.

 \end{theorem}

\begin{theorem}{\rm (\cite{Stepanets_UMZh2001_8}, \cite[Гл.~11]{Stepanets_M2005})}\label{Applic_In_Th}
 Нехай $ f\in {\mathcal S}^p $, $ p>0$, система $ \psi =\{\psi (\bf k)\}_{{\bf k}\in {\mathbb Z}^d}$
 задовольняє умову \eqref{w5}  і
 $$
 \lim _{k\to \infty }\varepsilon_k^{-1}E_k(f)_{_{\scriptstyle \psi ,{\mathcal S}^p}}=0.
 $$
Тоді для того, щоб виконувалося включення $ f\in L^{\psi }{\mathcal S}^p$  необхідно і достатньо, щоб збігався ряд
 $$
 \sum
 _{k=1}^{\infty }(\varepsilon_k^{-p} -
 \varepsilon_{k-1}^{-p})E_k^p(f)_{_{\scriptstyle \psi ,{\mathcal S}^p}}.
 $$
 Якщо цей ряд збігається, то при довільному  $n\in {\mathbb N},$ справджується рівність
 $$
 E_n^p(f^{\psi })_{\psi
 ,p}=\varepsilon_n^{-p}E_n^p(f)_{_{\scriptstyle \psi ,{\mathcal S}^p}} + \sum _{k=n+1}^{\infty
 }(\varepsilon_k^{-p}-\varepsilon_{k-1}^p)E_{k}^p(f)_{_{\scriptstyle \psi ,{\mathcal S}^p}},
 $$
 де  $ \varepsilon_k $ -- елементи характеристичної послідовності  $ \varepsilon(\psi )$.

 \end{theorem}

Дані теореми встановлюють зв'язок між найкращим  наближенням довільної функції з простору ${\mathcal S}^p$
і властивостями її $\psi$-похідних.  Вони були отримані О.\,І.~Степанцем в роботі \cite{Stepanets_UMZh2001_8} (див. також
\cite[Гл.~11]{Stepanets_M2005}).
Подібні твердження в теорії наближень прийнято називати прямими та оберненими теоремами.
Теорема \ref{Applic_Direct_Th} показує як швидкість спадання до нуля найкращих наближень
$\psi$-похідних функції впливає на оцінки її найкращих наближень.  Теорема \ref{Applic_In_Th} в цьому сенсі
є оберненою: у ній за властивостями найкращого наближення функції робиться висновок про існування її $\psi$-похідних та
 вказуються оцінка їх найкращих наближень в ${\mathcal S}^p$.

Зазначимо, що у роботах  \cite{Savchuk_Shidlich_UMZh2007, Savchuk_Shidlich_2014} та \cite{Chaichenko_Savchuk_Shidlich_2020}
встановлено, зокрема, аналогічні прямі та обернені теореми наближення функцій з просторів  ${\mathcal S}^p$
різними лінійними методами підсумовування рядів Фур'є (методами Фейєра, Зигмунда, Абеля-Пуассона тощо). У роботах \cite{Timan_Shavrova_2007, Timan_Shavrova_2009} та \cite[Гл.~3]{Timan_M2009}) отримано оцінки відхилень полігармонічних функцій $u_n(r,x)$, $0\le r<1$, $x\in {\mathbb T}^1$, порядку $n$ в крузі одиничного радіуса з центром в початку координат від їх граничних функцій $f(x)$   в метриках просторів ${\mathcal S}^p$ через найкращі наближення граничних функцій в тих же метриках.



\subsection{Найкращі наближення, найкращі $n$-членні наближення та тригонометричні поперечники
класів $L_q^{\psi }$ в просторах ${\mathcal S}^p$ } Нехай $f$ -- довільна функція з простору ${\mathcal S}^p$, $n$ -- будь-яке натуральне число і $\gamma_n$ -- довільний набір з $n$ векторів ${\bf k}\in {\mathbb Z}^d$.
Розглянемо тригонометричні поліноми
 \begin{equation}\label{w14a}
    P_{\gamma_n} = \sum\limits_{{\bf k}\in \gamma_n}
c_{\bf k} {\mathrm e}^{{\mathrm i}({\bf k},\cdot)}\quad \mbox{\rm та}\quad S_{\gamma_n}(f)=
  \sum\limits_{{\bf k}\in \gamma_n}\widehat{f}({\bf k}){\mathrm e}^{{\mathrm i}({\bf k},\cdot)},
 \end{equation}
де $c_{\bf k}$ -- будь-які комплексні числа, а $\widehat{f}({\bf k})$ -- коефіцієнти Фур'є функції $f$.

Величину
 \begin{equation}\label{Best_Approx}
  E_{\gamma_n}(f)_{_{\scriptstyle {\mathcal S}^p}} \!=\!\!\inf_{c_{\bf k}\in {\mathbb C}}
  \|f - P_{\gamma_n}\|_{_{\scriptstyle {\mathcal S}^p}}
 \end{equation}
називають найкращим наближенням функції $f$ довільними $n$-членними поліномами, що відповідають набору $\gamma_n$, а величину
 \begin{equation}\label{Approx_Fourier_Sums}
  {\mathscr E}_{\gamma_n}(f)_{_{\scriptstyle {\mathcal S}^p}}\!=\|f -\! S_{\gamma_n}(f) \|_{_{\scriptstyle {\mathcal S}^p}}.
 \end{equation}
-- наближенням функції  $f$ частинною сумою Фур'є, що відповідає набору $\gamma_n$.

Якщо ${\mathfrak N}$ -- деяка підмножина з ${\mathcal S}^p$, то  через
 $E_{\gamma_n}({\mathfrak N})_{_{\scriptstyle {\mathcal S}^p}}$ та
 ${\mathscr E}_{\gamma_n}({\mathfrak N})_{_{\scriptstyle {\mathcal S}^p}}$ позначають точні верхні межі
 величин  (\ref{Best_Approx}) та (\ref{Approx_Fourier_Sums}) по множині ${\mathfrak N}$, тобто,
  \begin{equation}\label{new15}
   E_{\gamma_n}({\mathfrak N})_{_{\scriptstyle {\mathcal S}^p}} = \sup\limits_{f\in {\mathfrak N}}
   E_{\gamma_n}(f)_{_{\scriptstyle {\mathcal S}^p}}\quad \mbox{\rm та}\quad
   {\mathscr E}_{\gamma_n}({\mathfrak N})_{_{\scriptstyle {\mathcal S}^p}} = \sup\limits_{f\in {\mathfrak N}}
   {\mathscr E}_{\gamma_n}(f)_{_{\scriptstyle {\mathcal S}^p}}.
 \end{equation}
Характеристики
\begin{equation}\label{new16}
   {\mathscr D}_n({\mathfrak N})_{_{\scriptstyle {\mathcal S}^p}} = \inf\limits_{\gamma_n}
  E_{\gamma_n}({\mathfrak N})_{_{\scriptstyle {\mathcal S}^p}} \quad \mbox{\rm та}\quad
   {\mathscr D}_n^\perp({\mathfrak N})_{_{\scriptstyle {\mathcal S}^p}} = \inf\limits_{\gamma_n}
  {\mathscr E}_{\gamma_n}({\mathfrak N})_{_{\scriptstyle {\mathcal S}^p}}.
 \end{equation}
називають відповідно тригонометричним (базисним)   та проєкційним поперечниками порядку  $n$ множини ${\mathfrak N}$
 в просторі ${\mathcal S}^p$.

Величини
  \begin{equation}\label{w12}
   \sigma_n(f)_{_{\scriptstyle {\mathcal S}^p}}=\inf_{\gamma _n} E_{\gamma_n}(f)_{_{\scriptstyle {\mathcal S}^p}}
    \quad \mbox{\rm та}\quad
    \sigma_n({\mathfrak N})_{_{\scriptstyle {\mathcal S}^p}}=\sup\limits _{f\in {\mathfrak N}}
    \sigma_n(f)_{_{\scriptstyle {\mathcal S}^p}}
   \end{equation}
називають найкращим $n$-членним (тригонометричним) наближенням відповідно функції $f$ та множини ${\mathfrak N}$ в просторі ${\mathcal S}^p$.

В ролі множини ${\mathfrak N}$ розглядаємо  множину  $L_q^{\psi }$   $\psi$-інтегралів всіх функцій з
одиничної кулі простору ${\mathcal S}^q$, $0<p,q<\infty$, за умов, що гарантують   вкладення   $L_q^{\psi }\subset {\mathcal S}^p$.

\begin{theorem}[{\cite[Гл.~11]{Stepanets_M2005}, \cite{Stepanets_UMZh2006_1}}]\label{Applic_q<p}
    Нехай   $0<q\le p$, і $ \psi =\{\psi ({\bf k})\}_{{\bf k}\in {\mathbb Z}^d} $ -- система чисел, підпорядкована умовам
    \eqref{w5}      і \eqref{b15}.
 Тоді для будь-якого  $n \in {\mathbb N}$ і
    для довільного набору  $\gamma_n$ із $n$ різних натуральных чисел справджуються рівності
 \begin{equation}\label{w30}
        E_{\gamma_n}(L_q^{\psi })_{_{\scriptstyle {\mathcal S}^p}}=
        {\mathscr E}_{\gamma_n}(L_q^{\psi })_{_{\scriptstyle {\mathcal S}^p}}
        =\widetilde{\psi}_{\gamma_n}(1),
    \end{equation}
 де $\widetilde{\psi}_{\gamma_n}(1)$ -- перший  член послідовності
 $\widetilde{\psi}_{\gamma_n}=\{\widetilde{\psi}_{\gamma_n}(k) \}^{\infty}_{k=1}$, яка є спадною перестановкою
 системи чисел $\{|\psi_{\gamma_n}({\bf k})|\}_{{\bf k}\in {\mathbb Z}^d}$,
\begin{equation}\label{w31}
 \psi_{\gamma_n}({\bf k})=\left\{\begin{matrix}0,\quad\hfill & {\bf k}\in\gamma_n,\\  \psi({\bf k}),\quad\hfill & {\bf k}\overline{\in}
 \gamma_n\end{matrix};\right.
 \end{equation}
при всіх $n \in {\mathbb N}$ виконуються рівності
\begin{equation}\label{w32}
   {\mathscr D}_n(L_q^{\psi })_{_{\scriptstyle {\mathcal S}^p}} =
   {\mathscr D}_n^\perp(L_q^{\psi })_{_{\scriptstyle {\mathcal S}^p}}=
   \widetilde{\psi}_{n+1},
 \end{equation}
\begin{equation}\label{w19}
 \sigma_n^p(L_q^{\psi})_{_{\scriptstyle {\mathcal S}^p}}=\mathop {\sup }\limits
 _{s>n}(s-n)(\sum _{k=1}^s\widetilde{\psi} _k^{-q})^{-\frac
 p{q}}=(s^*-n)(\sum _{k=1}^{s^*}\widetilde{\psi}_k^{-q})^{-\frac p{q}},
  \end{equation}
 в яких $ \widetilde{\psi} =\{\widetilde{\psi} _k\}_{k=1}^{\infty } $ -- спадна перестановка системи чисел $|\psi({\bf k})|$, ${{\bf k}\in {\mathbb Z}^d}$, і  $s^*$ -- деяке натуральне число.
\end{theorem}

\begin{theorem}[{\cite[Гл.~11]{Stepanets_M2005}, \cite{Stepanets_UMZh2006_1}}]\label{Applic_q>p}
    Нехай   $0<p<q$, і $ \psi =\{\psi ({\bf k})\}_{{\bf k}\in {\mathbb Z}^d} $ -- система чисел,
    яка задовольняє умови  \eqref{w5} та
    \begin{equation}\label{w21}
      \sum _{{\bf k}\in {\mathbb Z}^d}|\psi ({\bf k})|^{\frac {pq}{q-p}}<\infty.
    \end{equation}
Тоді для будь-якого  $n \in {\mathbb N}$ і  для довільного набору  $\gamma_n$ із $n$ різних натуральных чисел справджуються рівності
 \begin{equation}\label{w33}
   E_{\gamma_n}(L_q^{\psi })_{_{\scriptstyle {\mathcal S}^p}}=
  {\mathscr E}_{\gamma_n}(L_q^{\psi })_{_{\scriptstyle {\mathcal S}^p}} =\Big( \sum_{k=1}^{\infty} \,(\bar
 {\psi}_{\gamma_n}(k))^{\frac{p\,q}{q-p}}\Big)^{\frac{q-p}{p\,q}},
 \end{equation}
де  $\widetilde{\psi}_{\gamma_n}=\{\widetilde{\psi}_{\gamma_n}(k) \}^{\infty}_{k=1}$ -- спадна перестановка
 системи чисел $\{|\psi_{\gamma_n}({\bf k})|\}_{{\bf k}\in {\mathbb Z}^d}$;  при всіх $n \in {\mathbb N}$ виконуються рівності
\begin{equation}\label{w34}
   {\mathscr D}_n(L_q^{\psi })_{_{\scriptstyle {\mathcal S}^p}} =
   {\mathscr D}_n^\perp(L_q^{\psi })_{_{\scriptstyle {\mathcal S}^p}}=
   \Big( \sum_{k=n+1}^{\infty} \widetilde{\psi}_k^{\frac{p\,q}{q-p}}
  \Big)^{\frac{q-p}{p\,q}},
 \end{equation}
\begin{equation}\label{w23}
\sigma_n^p(L_q^{\psi })_{_{\scriptstyle {\mathcal S}^p}}\!=\!\bigg((s^*-n)^\frac q{q-p} \Big(\sum\limits_{k=1}^{s^*} \widetilde{\psi}_k^{-q}
  \Big)^ {-\frac p{q-p}}\!\!+\!\!
 \sum\limits_{k=s^*+1}^\infty  \widetilde{\psi}_k^\frac
 {pq}{q-p}\bigg)^\frac{q-p}{pq}, 
  \end{equation}
 в яких $ \widetilde{\psi} =\{\widetilde{\psi} _k\}_{k=1}^{\infty } $ -- спадна перестановка системи чисел $|\psi({\bf k})|$, ${{\bf k}\in {\mathbb Z}^d}$ і  $s^*$ -- деяке натуральне число.

\end{theorem}

Рівності  (\ref{w19}) та (\ref{w23}) отримано О.\,І.~Степанцем відповідно у роботах  \cite{Stepanets_UMZh2001_8} (у частинному випадку, коли $p=q$, в \cite{Stepanets_Preprint_2000, Stepanets_UMZh2001_3}) та \cite{Stepanets_Preprint_2001}
(див. також  \cite[Гл.~11]{Stepanets_M2005} та \cite{Stepanets_UMZh2006_1}). Вони, як і рівності (\ref{w30}), (\ref{w32}), (\ref{w33}) та (\ref{w34}),  випливають із відповідних результатів для більш загальних просторів $S^{p}_\varphi$  ($S^{p,\,\mu}_\varphi$) \cite{Stepanets_UMZh2006_1}. Зазначимо, що низку цікавих результатів в даному напрямку також  отримано в роботах \cite{Stepanets_Shydlich_UMZh2003, Stepanets_UMZh2003_10, Stepanets_Shidlich_JAT_2010, Stepanets_Shidlich_IZV_2010, Fuchang_Gao_2010, Shidlich_Chaichenko_LM_2014, Shidlich_Chaichenko_Orlicz_lM_2015} та ін.


\subsection{Поперечники за Колмогоровим класів $L_p^{\psi }$}\label{WidthSubsection}

Розглянемо випадок, коли апроксимуючі поліноми будуються по наборах $\gamma_{_{\scriptstyle \delta_n}}$, які визначаються через  елементи
характеристичної послідовності  $g(\psi )$  системи $ \psi$ ($\gamma_{_{\scriptstyle \delta_n}}{=}g_n(\psi)$). У цьому випадку поліноми (\ref{w14a}) мають вигляд
 \[
 P_{g_n(\psi)}=\sum_{{\bf k}\in g_n(\psi)}
c_{\bf k} {\mathrm e}^{{\mathrm i}({\bf k},\cdot)}
\] і
\begin{equation}\label{w6}
 S_{g_n(\psi)}(f) =\sum_{{\bf k}\in g_n(\psi)}\widehat f({\bf k}) {\mathrm e}^{{\mathrm i}({\bf k},{\bf \cdot})}=:S_n(f)_{\psi},\quad S_0(f)_{\psi}:=0,
  \end{equation}
величини (\ref{Best_Approx}) збігаються з величинами (\ref{b21}), а величини (\ref{Approx_Fourier_Sums}) -- з величинами
 \[%
  \mathscr {E}_n(f)_{_{\scriptstyle \psi ,{\mathcal S}^p}}=\|f- S_{n-1}(f)_{\psi}\|_{_{\scriptstyle {\mathcal S}^p}}.
  \]

Якщо ${\mathfrak N}$ -- деяка підмножина з $L^{\psi }{\mathcal S}^p$, то  покладаємо
 \[
   E_{n}({\mathfrak N})_{_{\scriptstyle \psi ,{\mathcal S}^p}} = \sup\limits_{f\in {\mathfrak N}}
   E_{\gamma_n}(f)_{_{\scriptstyle \psi ,{\mathcal S}^p}}\quad \mbox{\rm і}\quad
   {\mathscr E}_{\gamma_n}({\mathfrak N})_{_{\scriptstyle \psi ,{\mathcal S}^p}} = \sup\limits_{f\in {\mathfrak N}}
   {\mathscr E}_{\gamma_n}(f)_{_{\scriptstyle \psi ,{\mathcal S}^p}}.
 \]

 \begin{theorem}[{\cite{Stepanets_UMZh2001_3}, \cite[Гл.~11]{Stepanets_M2005}}]\label{Applic_Psi_q<p}
    Нехай  $0<q\le p$  і $ \psi =\{\psi ({\bf k})\}_{{\bf k}\in {\mathbb Z}^d} $ -- система чисел, підпорядкована умовам  \eqref{w5}
    та \eqref{b15}.  Тоді для будь-якого  $n \in {\mathbb N}$ справджуються рівності
\begin{equation}\label{w18}
 E_n(L_q^{\psi })_{_{\scriptstyle \psi ,{\mathcal S}^p}}={\mathscr{E}}_n(L_q^{\psi })_{_{\scriptstyle \psi ,{\mathcal S}^p}}=\varepsilon_n,
  \end{equation}
де  $\varepsilon_n$  -- члени характеристичної послідовності системи $\varepsilon(\psi)$.
\end{theorem}

\begin{theorem}[{\cite[Гл.~11]{Stepanets_M2005}, \cite{Stepanets_UMZh2006_1}}]\label{Applic_Psi_q>p}
    Нехай   $0<p<q$  і $ \psi =\{\psi ({\bf k})\}_{{\bf k}\in {\mathbb Z}^d} $ -- система чисел,
    яка задовольняє умови  \eqref{w5} та \eqref{w21}. Тоді для будь-якого  $n \in {\mathbb N}$ справджуються рівності
\begin{equation}\label{w22}
E_n(L_q^{\psi})_{_{\scriptstyle \psi ,{\mathcal S}^p}}={\mathscr{E}}_n(L_q^{\psi })_{_{\scriptstyle \psi ,{\mathcal S}^p}}=
\bigg(\sum\limits_{k=\delta_{n-1}+1}^\infty
\widetilde{\psi}_k^{\frac{pq}{q-p}}\bigg)^{\frac{q-p}{pq}},
\end{equation}
в яких  $ \widetilde{\psi} =\{\widetilde{\psi} _k\}_{k=1}^{\infty } $ -- спадна перестановка системи чисел $|\psi({\bf k})|$, ${{\bf k}\in {\mathbb Z}^d}$, а
 $\delta_n$ -- члени характеристичної послідовності $\delta(\psi)$.

\end{theorem}


 Нехай
${\mathscr G}_n$ -- множина всіх $n$-вимірних лінійних підпросторів $G_n$  в ${\mathcal S}^p$, $n\in {\mathbb N}$, і
 \[
   d_n(L_p^{\psi })_{_{\scriptstyle {\mathcal S}^p}}= \inf\limits_{G_n\in\, {\mathscr G}_n} \sup\limits_{f\in L_p^{\psi }}
      \inf\limits _{u\in G_n}\|f-u \|_{_{\scriptstyle {\mathcal S}^p}}, \ \
      d_0(L_p^{\psi })_{_{\scriptstyle {\mathcal S}^p}}:=\sup \limits _{f\in L_p^{\psi }}\|f\|_{_{\scriptstyle {\mathcal S}^p}}, 
 \]
   -- поперечники за  Колмогоровим класів $L_p^{\psi }$  в просторі ${\mathcal S}^p$.

\begin{theorem}[{\cite{Stepanets_UMZh2001_3}, \cite[Гл.~11]{Stepanets_M2005}}]\label{Applic_Kolmogorov_W} Нехай $ p\in [1,
\infty )$ і $ \psi =\{\psi ({\bf k})\}_{{\bf k}\in {\mathbb Z}^d} $ -- система чисел, підпорядкована умовам  \eqref{w5} та \eqref{b15}. Тоді при довільних   $ n\in {\mathbb N}$ виконуються рівності
              $$
              d_{\delta _{n-1}}(L_p^{\psi })_{_{\scriptstyle {\mathcal S}^p}}=  d_{\delta _{{n-1}}+1}(L_p^{\psi })_{_{\scriptstyle {\mathcal S}^p}}=\ldots
              $$
         \begin{equation}\label{w29}
          = d_{\delta _n-1}(L_p^{\psi })_{_{\scriptstyle {\mathcal S}^p}}= {\mathscr{E}}_n(L_p^{\psi })_{_{\scriptstyle \psi ,{\mathcal S}^p}}=\varepsilon_n,
         \end{equation}
 де $ \varepsilon_n$ та $ \delta _n $ -- члени характеристичних  послідовностей     $ \varepsilon(\psi )$ та $
\delta (\psi ),$ відповідно.
 \end{theorem}

Теореми \ref{Applic_Psi_q<p} та  \ref{Applic_Kolmogorov_W} доведено  в роботах \cite{Stepanets_Preprint_2000, Stepanets_UMZh2001_3}, а теорему \ref{Applic_Psi_q>p} -- в \cite{Stepanets_Preprint_2001} (див. також  \cite[Гл.~11]{Stepanets_M2005} та \cite{Stepanets_UMZh2006_1}).


\subsection{Апроксимативні характеристики класів ${\mathcal F}_{q,r}^{\psi}$
 в просторах ${\mathcal S}^{p}$}\label{F^psi results}

 Класи ${\mathcal F}_{q,r}^{\psi}$, означені в п. \ref{psi-derivatives}.5, збігаються з множинами $L_q^{\psi^* }$ у випадку, коли система $\psi^*=\{\psi^*({\bf k})\}_{{\bf k}\in {\mathbb Z}^d}$ задовольняє рівності $\psi^*({\bf k})=\psi(|{\bf k}|_{r})$, ${\bf k}\in {\mathbb Z}^d$.
 Тому для них справджуються наведені вище теореми \ref{Applic_q<p}--\ref{Applic_Kolmogorov_W}.
 У цьому підрозділі наведено низку тверджень, у яких  встановлено точні порядкові оцінки апроксимативних характеристик класів  ${\mathcal F}_{q,r}^{\psi}$. Дані результати були використані, зокрема, до розв'язання задач наближення  у просторах Лебега $L_p({\mathbb T}^d)$ (див. далі п. \ref{Apll_Lp}.2).

Позначимо через $B$ множину всіх монотонно спадних до нуля при $t\to\infty$ функцій $\psi(t)$, $t\ge 1$, які задовольняють так звану $\Delta_2$-умову:
\begin{equation}\label{5.m9}
{\psi(t)}\le K {\psi(2t)},\quad K>0.
\end{equation}
При формулюванні результатів також вважаємо, що  при всіх достатньо великих $m$ (більших, ніж деяке додатне число $k_0$)
число $|\widetilde{\Delta}_{m,r}^d|$ елементів множини  $
\widetilde{\Delta}_{m,r}^d:=\{{\bf k}\in {\mathbb Z}^d:|{\bf k}|_{r}\le  m,\ \ \ m=0,1,\ldots\}$ задовольняє співвідношення
 \begin{equation}\label{a2.212}
   {M_r(m-c_1)^d}<|\widetilde{\Delta}_{m,r}^d|\le {M_r(m+c_2)^d},
 \end{equation}
де $M_r$, $c_1$ та  $c_2$ --- деякі додатні сталі.

Зазначимо, що у випадку, коли $r=\infty$, співвідношення (\ref{a2.212}) виконується і $M_\infty={\rm vol}\{{\bf k}\in {\mathbb R}^d:|{\bf k}|_{\infty}\le1\}=2^d$, якщо ж $r=1$, то $M_1={\rm vol}\{{\bf k}\in {\mathbb R}^d:|{\bf k}|_{1}\le  1\}=2^d/d!$.

\begin{proposition}[\cite{Shidlich_Zb2011,Shidlich_2016}]\label{sigma_n B} Нехай $d\in {\mathbb N}$, $0<r\le\infty$, $0<p,\,q<\infty$,  виконуються умова \eqref{a2.212},  і   функція $\psi_1(t):=\psi^p(t)$ належить  множині $B$, а при $0<p<q$, крім того, при всіх $t$, більших
деякого числа $t_0$, є опуклою  та задовольняє умову
   \begin{equation}\label{1.2c121}
         t|\psi'_1(t)|/{\psi_1(t)}\ge K_0>\beta,
     \end{equation}
де $\psi'_1(t):=\psi'_1(t+)$, $\beta=d(1/p-1/q)$. Тоді
\begin{equation}\label{ss1}
\sigma_n({\mathcal F}_{q,r}^{\psi})_{_{\scriptstyle {\mathcal S}^{p}}}\asymp
 {\psi(n^\frac 1d)}{\,n^{\frac 1p\,-\frac 1q}}.
\end{equation}
\end{proposition}

  Для додатних послідовностей  $\alpha(n)$ та $\beta(n)$ вираз  ``$a(n)\asymp b(n)$''   означає, що існують сталі
  $0<K_1<K_2$ такі, що при всіх $n\in\mathbb{N}$, $\ \alpha(n)\le K_2\beta(n)$ (в цьому випадку пишемо
  ``$\alpha(n)\ll\beta(n)$'') і $\alpha(n)\ge K_1\beta(n)$.

Враховуючи вигляд оцінки \eqref{ss1} і те, що умова \eqref{a2.212} виконується, зокрема, при $r=1$ та $r=\infty$, з даного твердження випливає такий наслідок.

\begin{corollary}\label{sigma_n B col} Нехай  $d\in {\mathbb N}$, $0<p,\,q<\infty$  і   функція $\psi$ задовольняє умови твердження \ref{sigma_n B}. Тоді для довільного $r\in [1,\infty]$ має місце оцінка \eqref{ss1}.
\end{corollary}

Дійсно, для довільних чисел $r\in [1,\infty]$, $0<q<\infty$ і будь-якої додатної спадної функції $\psi$ мають місце вкладення
 \begin{equation}\label{er1}
{\mathcal F}_{q,1}^{\psi}\subset {\mathcal F}_{q,r}^{\psi}\subset {\mathcal F}_{q,\infty}^{\psi}.
\end{equation}
Тому якщо виконуються умови наслідку \ref{sigma_n B col}, то для $r\in [1,\infty]$
 $$
 {\psi(n^\frac 1d)}{\,n^{\frac 1p-\frac 1q}}\ll \sigma_n({\mathcal F}_{q,1}^{\psi})_{_{\scriptstyle {\mathcal S}^{p}}}\le \sigma_n({\mathcal F}_{q,r}^{\psi})_{_{\scriptstyle {\mathcal S}^{p}}}\le
\sigma_n({\mathcal F}_{q,\infty}^{\psi})_{_{\scriptstyle {\mathcal S}^{p}}}\ll
 {\psi(n^\frac 1d)}{\,n^{\frac 1p-\frac 1q}}.
 $$


\begin{proposition}[\cite{Shidlich_Zb2011,Shidlich_2016}]\label{D_n B} Нехай    $d\in {\mathbb N}$, $0<r\le\infty$, $0<p, \,q<\infty$ і виконується умова \eqref{a2.212}. Тоді

1) якщо $0<q\le p<\infty$, то для довільної функції $\psi\in B$ має місце формула
    \begin{equation}\label{f311}
{\mathscr D}_n({\mathcal F}_{q,r}^{\psi})_{_{\scriptstyle {\mathcal S}^{p}}}={\mathscr D}_n^\perp({\mathcal F}_{q,r}^{\psi})_{_{\scriptstyle {\mathcal S}^{p}}}\asymp \psi(m_n)\asymp \psi(n^{1/d}),
\end{equation}

2) якщо ж $0<p<q<\infty$, а функція $\psi^p$ належить  $B$ і при всіх $t$, більших ніж
деяке число $t_0$, є опуклою  та задовольняє умову \eqref{1.2c121} при $\beta=d(1/p-1/q)$, то
    \begin{equation}\label{f31122}
{\mathscr D}_n({\mathcal F}_{q,r}^{\psi})_{_{\scriptstyle {\mathcal S}^{p}}}={\mathscr D}_n^\perp({\mathcal F}_{q,r}^{\psi})_{_{\scriptstyle {\mathcal S}^{p}}}\asymp
\psi(n^\frac 1d)n^{\frac 1p-\frac 1q}.
\end{equation}
\end{proposition}

 Твердження \ref{sigma_n B} та \ref{D_n B} доведено в  \cite{Shidlich_Zb2011,Shidlich_2016}.
 У роботах \cite{Shidlich_Zb2013, Shidlich_Zb2014} аналогічні оцінки величин $\sigma_n({\mathcal F}_{q,r}^{\psi})_{_{\scriptstyle {\mathcal S}^{p}}}$,  ${\mathscr D}_n({\mathcal F}_{q,r}^{\psi})_{_{\scriptstyle {\mathcal S}^{p}}}$ та
 ${\mathscr D}_n^\perp({\mathcal F}_{q,r}^{\psi})_{_{\scriptstyle {\mathcal S}^{p}}}$ отримано у випадку, коли функція
 $\psi$ спадає до нуля швидше довільної степеневої функції. У випадку, коли $d=1$, подібні оцінки також
 можна отримати як наслідки з  відповідних оцінок даних величин у просторах $S^p_\varphi$ \cite{Shydlich_UMZh2009}. Співставляючи результати даних тверджень,
 робимо висновок, що
  $
      \sigma_n({\mathcal F}_{q,r}^{\psi})_{_{\scriptstyle {\mathcal S}^{p}}}\asymp {\mathscr D}_n({\mathcal F}_{q,r}^{\psi})_{_{\scriptstyle {\mathcal S}^{p}}}={\mathscr D}_n^\perp({\mathcal F}_{q,r}^{\psi})_{_{\scriptstyle {\mathcal S}^{p}}}$ при  $0<p<q$, і
 $
 \sigma_n({\mathcal F}_{q,r}^{\psi})_{_{\scriptstyle {\mathcal S}^{p}}}=o\Big({\mathscr D}_n({\mathcal F}_{q,r}^{\psi})_{_{\scriptstyle {\mathcal S}^{p}}}\Big)$, $n\to\infty$, при
 $0<q\le p$.


\vskip 2mm
\subsection{Деякі наслідки для просторів $L_p({\mathbb T}^d)$}\label{Apll_Lp}
{\bf \ref{Apll_Lp}.1.} Нехай $L_p=L_p({\mathbb T}^d),$ $p\in [1, \infty ),$ -- простір функцій $f\in L$ зі скінченною нормою  (\ref{Lp_norm}). Зв'язок між просторами  $L_p$ та ${\mathcal S}^p$ встановлює відома теорема Гаусдорфа--Юнга (див., наприклад, \cite{Temlyakov_B1993}),
з якої випливає, що при  $p\in (1,2]$ і $\frac 1p+\frac 1{p'}=1$ мають місце формули
\begin{equation}\label{w38}
 L_p\subset {\mathcal S}^{p'}\ \ \ \ \mbox {і} \ \ \ \ \|f\|_{{\mathcal S}^{p'}}\le \|f\|_{L_{p}}, 
  \end{equation}
\begin{equation}\label{w39}
 {\mathcal S}^{p'}\subset L_{p}\ \ \ \ \mbox {і} \ \ \ \ \|f\|_{L_{p}}\le \|f\|_{_{\scriptstyle {\mathcal S}^{p'}}}. 
  \end{equation}
Зокрема, при  $p=p'=2 $ виконуються рівності
 \begin{equation}\label{w40}
 L_2={\mathcal S}^2 \ \ \ \ \mbox{і} \ \ \ \ \|\cdot\|_{L_2}=\|\cdot
 \|_{{\mathcal S}^2}. 
   \end{equation}
Отже,  теореми, доведені для просторів ${\mathcal S}^p,$ містять певну інформацію і
для просторів $L_p,$ яка є найбільш повною внаслідок  (\ref{w40}), у випадку, коли $p=2.$

\begin{corollary}[{\cite[Гл.~11]{Stepanets_M2005}, \cite{Stepanets_UMZh2006_1}}]\label{Applic_Psi_p=2}
    Нехай  $ \psi =\{\psi ({\bf k})\}_{{\bf k}\in {\mathbb Z}^d} $ -- система чисел, підпорядкована умовам  \eqref{w5}
    та \eqref{b15}.  Тоді для будь-якого  $n \in {\mathbb N}$ справджуються рівності
             $$
              d_{\delta _{n-1}}(L_2^{\psi })_{_{\scriptstyle L_2}}=
              d_{\delta _{{n-1}}+1}(L_2^{\psi })_{_{\scriptstyle L_2}}=\ldots
              = d_{\delta _n-1}(L_2^{\psi })_{_{\scriptstyle L_2}}
              $$
         \begin{equation}\label{p=2}
          =E_n(L_2^{\psi })_{_{\scriptstyle \psi ,L_2}}=
          {\mathscr{E}}_n(L_2^{\psi })_{_{\scriptstyle \psi ,L_2}}=\varepsilon_n,
         \end{equation}
 де $ \varepsilon_n$ та $ \delta _n $ -- члени характеристичних  послідовностей     $ \varepsilon(\psi )$ та $
\delta (\psi ),$ означених в \eqref{b17i}.
\end{corollary}

Рівності  (\ref{p=2}) в одновимірному випадку ($d=1$) для класів Соболєва $W^r_2$ (при $\psi(k)=k^{-r}$, $r\in {\mathbb N}$)
отримав у 1936 році  А.\,М.~Колмогоров \cite{Kolmogoroff_1936}, який започаткував новий напрям в теорії наближень, пов'язаний з
дослідженням поперечників різних функціональних класів.

Як випливає з (\ref{p=2}),  у просторі $L_2$ поперечники множин $L_2^{\psi }$ реалізують
 суми Фур'є (\ref{w6}), номери гармонік яких належать областям $g_n(\psi)$ вигляду \eqref{b17iq}.

 Зазначимо, що відомі класи диференційовних функцій Соболєва отримуються з  $L^{\psi } U_{_{\scriptstyle L_p}}$, де $U_{_{\scriptstyle L_p}}$ -- одинична куля простору $L_p$, якщо
 $\psi(\bf k)$ має вигляд   (\ref{w35}) з
  \begin{equation}\label{w45}
 \psi
 _j(k_j)=\left \{ \begin{matrix} 1, \hfill & k_j=0, \hfill \\  ({\mathrm i}k_j)^{-r_j}, \hfill & k_j\not =0,
 \hfill\end{matrix}\right.\quad j=\overline {1,d},\
 r_j\in {\mathbb R}. 
     \end{equation}

Нехай $d=2$, а послідовності $\psi _1(k_1)$ та $\psi _2(k_2)$ означені рівностями  (\ref{w45}) за умови  $r_1=r_2=r>0.$
У цьому випадку класи $L^{\psi } U_{_{\scriptstyle L_p}}$  з точки зору знаходження поперечників
вперше розглядалися К.\,І.~Бабенком в  \cite{Babenko_1960_5, Babenko_1960_2}, який цьому випадку фактично отримав співвідношення (\ref{p=2}).

В цій ситуації  характеристична послідовність $\varepsilon(\psi )$ складається
з елементів $\varepsilon_n=n^{-r},$ $n\in {\mathbb N}, $ множини $g_n(\psi)$ -- множини
векторів ${\bf k}=(k_1,k_2)\in {\mathbb Z}^2$, які задовольняють умову
 $$
 k_1'k_2'\le n,
 $$
 де
 $$
 k_j'=\left \{ \begin{matrix} 1,
 \hfill & k_j=0, \hfill \\ |k_j|, \hfill & k_j\not =0 , \ \ j=1,2
 .\hfill \end{matrix}\right.
 $$
 Такі множини   з'явились у   \cite{Babenko_1960_5, Babenko_1960_2} і зараз їх
 прийнято називати
 гіперболічними хрестами.


{\bf \ref{Apll_Lp}.2.}  Для  $f\in L_1({\mathbb T}^d)$ позначимо через
 $\{{\bf k}_l\}_{l=1}^\infty=\{{\bf k}_l(f)\}_{l=1}^\infty$ будь-яку перестановку векторів ${\mathbb Z}^d$, що задовольняє умову
\begin{equation}\label{a2.3a1}
|\widehat{f}({\bf k}_1)|\ge |\widehat{f}({\bf k}_2)|\ge \ldots.
\end{equation}
 Для довільної функції $f\in X$, де $X$ -- один із просторів $L_p=L_p({\mathbb T}^d)$, $1\le p\le \infty$,
 або ${\mathcal S}^p={\mathcal S}^p({\mathbb T}^d)$, $0<p<\infty$, $d\in {\mathbb N}$, при будь-якому $n=0,1,\ldots$ розглянемо   величини
 \begin{equation}\label{a2.3a2}
  \|f-G_n(f)\|_{_{\scriptstyle X}}:=\Big\|f(\cdot)-\sum\limits_{l=1}^n \widehat{f}({\bf k}_l)
  {\mathrm e}^{{\mathrm i}({\bf k}_l,\cdot)}  \Big\|_{_{\scriptstyle X}},
 \end{equation}
\begin{equation}\label{a2.3a21}
   \sigma_n^\perp(f)_{_{\scriptstyle X}}:=\inf\limits_{\gamma_n}\Big\|f-\sum\limits_{{\bf k}\in \gamma_n}\widehat{f}({\bf k})
   {\mathrm e}^{{\mathrm i}({\bf k},\cdot)}\Big\|_{_{\scriptstyle X}},
\end{equation}
та
 \begin{equation}\label{a2.3a22}
    \sigma_n(f)_{_{\scriptstyle X}}:=\inf\limits_{\gamma_n, c_{\bf k}}\Big\|f-\sum\limits_{{ k}\in \gamma_n}c_{\bf k}{\mathrm e}^{{\mathrm i}({\bf k},\cdot)}\Big\|_{_{\scriptstyle X}},
  \end{equation}
де $\gamma_n$ -- довільний набір з $n$ векторів ${\bf k}\in {\mathbb Z}^d$, $c_{\bf k}$ -- будь-які комплексні числа.

Величину $(\ref{a2.3a2})$ називають наближенням функції $f$ за допомогою $n$-членних гріді (greedy) апроксимант,  величину  $(\ref{a2.3a21})$ --  найкращим $n$-членним ортогональним тригонометричним наближенням функції $f$, а величину $(\ref{a2.3a22})$ -- найкращим $n$-членним тригонометричним наближенням функції $f$.

Для довільної множини ${\mathfrak N}\subset X$ покладаємо
$$
\sigma_n^\perp({\mathfrak N})_{_{\scriptstyle X}}:=\sup\limits_{f\in {\mathfrak N}}\sigma_n^\perp(f)_{_{\scriptstyle X}} \quad \mbox{\rm і}\quad
\sigma_n({\mathfrak N})_{_{\scriptstyle X}}:=\sup\limits_{f\in {\mathfrak N}}\sigma_n(f)_{_{\scriptstyle X}}.
$$
Взагалі кажучи, величини (\ref{a2.3a2}) залежать від вибору перестановки, яка задовольняє  (\ref{a2.3a1}).
Тому для однозначеності позначення покладаємо
\begin{equation}\label{a2.3a25g}
G_n({\mathfrak N})_{_{\scriptstyle X}}:=\sup\limits_{f\in {\mathfrak N}}\inf\limits_{\{{\bf k}_l(f)\}_{l=1}^\infty}\Big\|f(\cdot)-\sum\limits_{l=1}^n \widehat{f}({\bf k}_l)
  {\mathrm e}^{{\mathrm i}({\bf k}_l,\cdot)}  \Big\|_{_{\scriptstyle X}}.
\end{equation}
Зрозуміло, що для довільної функції  $f\in L_p$,
\begin{equation}\label{a2.3a27}
\sigma_n(f)_{_{\scriptstyle L_p}}\le \sigma_n^\perp (f)_{_{\scriptstyle L_p}}\le \|f-G_n(f)\|_{_{\scriptstyle L_p}}.
\end{equation}
якщо ж $f\in {\mathcal S}^p$, то
\begin{equation}\label{a2.3a26}
\sigma_n(f)_{_{\scriptstyle {\mathcal S}^p}}=\sigma_n^\perp (f)_{_{\scriptstyle {\mathcal S}^p}}=\|f-G_n(f)\|_{_{\scriptstyle {\mathcal S}^p}}.
\end{equation}

Точні порядкові оцінки апроксимативних  характеристики класів ${\mathcal F}_{q,r}^{\psi}$
 в просторах ${\mathcal S}^{p}$, наведені в пункті \ref{F^psi results}, дали можливість застосувати їх
 до знаходження порядкових оцінок аналогічних характеристик у просторах
 $L_p$. Такі результати отримано в циклі робіт \cite{Shidlich_Zb2011, Shidlich_Zb2013, Shidlich_Zb2014, Shidlich_2016}, який охоплює
 множину всіх функцій $\psi$, починаючи від повільно спадних функцій і закінчуючи функціями, які спадають до нуля
 швидше геометричної прогресії.

\begin{theorem}[\cite{Shidlich_Zb2011, Shidlich_2016}] \label{Th3.1}
  Нехай  $1\le r\le \infty$, $1\le p<\infty$, $0<q<\infty$, функція  $\psi$ належить множині $B$, та у випадку, коли
  $p/{(p-1)}<q$,  при всіх $t$, більших деякого числа $t_0$,  є опуклою і задовольняє умову \eqref{1.2c121}
  з  $\beta=d(\frac 12-\frac 1q)$ при $1< p\le 2$ і з $\beta=d(1-\frac 1p-\frac 1q)$ при $2\le p<\infty$. Тоді
 \begin{equation}\label{1.2c122}
   G_n({\mathcal F}_{q,r}^{\psi})_{_{\scriptstyle L_p}}\asymp
   \sigma_n^\perp({\mathcal F}_{q,r}^{\psi})_{_{\scriptstyle L_p}}\asymp
   \left\{\begin{matrix}  {\psi(n^{\frac 1d})}{n^{\frac 12-\frac 1q}},\ \  \ 1\le p\le 2, \\ {\psi(n^{\frac 1d})}{n^{1-\frac 1p-\frac 1q}},\ \  \  2\,{\le}\, p\,{<}\infty\end{matrix},\right.
 \end{equation}
при всіх  $1\le p\le 2$
\begin{equation}\label{1.2c1221}
\sigma_n({\mathcal F}_{q,r}^{\psi})_{_{\scriptstyle L_p}}\asymp   {\psi(n^{\frac 1d})}{n^{\frac 12-\frac 1q}},
\end{equation}
і при всіх   $2<p<\infty$,
$$
 {\psi(n^{\frac 1d})}{n^{\frac 12-\frac 1q}}\ll \sigma_n({\mathcal F}_{q,r}^{\psi})_{_{\scriptstyle L_p}}\ll \psi(n^{\frac 1d}) n^{1-\frac 1p-\frac 1q}.
$$
\end{theorem}

\begin{theorem}[\cite{Shidlich_2016}]\label{Th3.2}
 Нехай  $1\le r\le \infty$, $2< p\le\infty$, $0<q< \infty$, функція  $\psi$ належить множині $B$ і при всіх $t$, більших деякого числа $t_0$,  є опуклою і задовольняє умову \eqref{1.2c121} з $\beta=d(1-\frac 1q)_+$, $(a)_+:=\max\{0,a\}.$ Тоді виконується порядкова рівність   \eqref{1.2c1221}.
 \end{theorem}

Зрозуміло, що умови в теоремах  \ref{Th3.1} та \ref{Th3.2} забезпечують вкладення  ${\mathcal F}_{q,r}^{\psi}\subset L_p$.
У випадку, коли $\psi(t)=t^{-s}$, $s>0$, аналогічні оцінки величин $\sigma_n({\mathcal F}_{q,r}^{\psi})_{_{\scriptstyle L_p}}$ та  $G_n({\mathcal F}_{q,r}^{\psi})_{_{\scriptstyle L_p}}$ раніше були отримані  в роботах
\cite{DeVore_Temlyakov_1995} та \cite{Temlyakov_Greedy_1998} відповідно. Наведені теореми показують, зокрема, що оцінки такого вигляду мають місце для більш широкої множини функцій. Наприклад, при  $0<q\le p/(p-1)$ умови теореми \ref{Th3.1} задовольняють функції  $\psi(t)=t^{-s}\ln^\varepsilon (t+e)$, де $s> 0$, $\varepsilon \in {\mathbb R}$, та $\psi(t)=\ln^\varepsilon (t+e)$,  $\varepsilon<0$. Умови теореми \ref{Th3.2}, а також теореми \ref{Th3.1} при $1<p/(p-1)<q$ задовольняють функції вигляду
 $\psi(t)=t^{-s}\ln^\varepsilon (t+e)$, де $\varepsilon \in {\mathbb R}$ і $s=s(p,q)$ -- деяке додатне число.


\section{Прямі та обернені теореми наближення майже періодичних функцій в просторах  $B {\mathcal S}^p$}\label{BSp_section}

 У цьому розділі сформульовано прямі та обернені теореми наближення майже періодичних функцій у просторах Безиковича-Степанця $B {\mathcal S}^p$  в термінах найкращих наближень функцій та  узагальнених модулів гладкості. Тематику досліджень в даному напрямку розвинуто в роботі авторів \cite{Serdyuk_Shidlich_2022}.

\subsection{Означення просторів  $B {\mathcal S}^p$}

Нехай   $ {B}^s$,  $ 1\le s<\infty$, -- простір усіх функцій, сумовних в степені $s$ на кожному скінченному інтервалі дійсної осі, у якому відстань визначається рівністю
 \[
 D_{_{\scriptstyle  B^s}}(f,g)=\Big(\mathop{\overline{\lim}}\limits_{T\to \infty}\frac 1{2T}\int_{-T}^T
 |f(x)-g(x)|^s {\mathrm d}x\Big)^{1/s}.
 \]
Нехай, далі, ${\mathscr T}$ -- множина всіх тригонометричних сум вигляду $\tau_N(x)=\sum_{k=1}^N a_k {\mathrm e}^{{\mathrm i} \lambda_kx}$, $N\in {\mathbb N}$, де  $\lambda_k\in {\mathbb R}$, $a_k\in {\mathbb C}$.

Функція  $f$ називається майже періодичною за Безиковичем порядку $s$ функцією (або $B^s$-м.п. функцією) і позначається $f\in B^s$-м.п.  \cite[Гл.~5, \S10]{Levitan_M1953}, \cite[Гл.~2, \S7]{Besicovitch_M1955}, якщо  існує така послідовність тригонометричних сум
$\tau_1, \tau_2, \ldots$ з множини   ${\mathscr T}$, що
 \[
 \lim_{N\to \infty} D_{_{\scriptstyle  B^s}}(f,\tau_N)=0.
 \]

Якщо $s_1\ge s_2\ge 1$, то (див., наприклад, \cite{Bredikhina_1968,  Bredikhina_1984}) $B^{s_1}$-м.п.$\subset B^{s_2}$-м.п.$\subset B$-м.п., де  $B$-м.п.$:=B^1$-м.п.   Для довільної  $B$-м.п. функції  $f$, існує середнє значення
 \[
 M\{f\}=\lim\limits_{T\to \infty} \frac 1 T \int_0^T f(x){\mathrm d}x.
 \]
Функція $M\{f(\cdot) {\mathrm e}^{-{\mathrm i}\lambda \cdot} \}$, $\lambda\in {\mathbb R}$, може набувати ненульові значення щонайбільше на зліченній множині.  Пронумерувавши елементи цієї множини в довільному порядку, отримуємо множину ${\mathscr S}(f)=\{\lambda_k\}_{k\in {\mathbb N}}$  показників Фур'є, яка називається спектром функції  $f$. Числа $A_{\lambda_k}=A_{\lambda_k}(f)=M\{f(\cdot) {\mathrm e}^{-{\mathrm i}\lambda_k \cdot} \}$ називаються коефіцієнтами Фур'є
функції $f$. Кожній функції $f\in B$-м.п.  зі спектром ${\mathscr S}(f)$  ставиться у відповідність ряд Фур'є вигляду
 $
 \sum_k A_{\lambda_k} {\mathrm e}^{ {\mathrm i}\lambda_k x}.
$
При цьому якщо  $f\in B_2$-м.п., то має місце рівність Парсеваля (див., наприклад, \cite[Гл.~2, \S9]{Besicovitch_M1955})
 \[
 M\{|f|^2\}=\sum_{k\in {\mathbb N}} |A_{\lambda_k}|^2.
 \]
При фіксованому  $ 1\le p<\infty$ розглянемо простори всіх функцій
$f\in B$-м.п., для яких є скінченною величина
 \begin{equation}\label{norm_Sp_BS}
 \|f\|_{_{\scriptstyle  p}}:=
 \|f\|_{_{\scriptstyle  B{\mathcal S}^p}} =\|\{A_{\lambda_k}(f)\}_{k\in {\mathbb N}}\|_{_{\scriptstyle  l_p({\mathbb N})}} =
 \Big(\sum_{k\in {\mathbb N}}|A_{\lambda_k}(f)|^p\Big)^{1/p}.
\end{equation}
Дані простори позначаються  $B{\mathcal S}^p$  і називаються просторами Безиковича-Степанця \cite{Serdyuk_Shidlich_2022}. За означенням  $B$-м.п. функції вважаються тотожними в  $B{\mathcal S}^p$, якщо вони мають однакові коефіцієнти Фур'є.

Простори $B {\mathcal S}^p$ є природним узагальненням як розглянутих вище просторів  $2\pi$-періодичних функцій ${\mathcal S}^p={\mathcal S}^p({\mathbb T}^1)$ (оскільки ${\mathcal S}^p\subset B {\mathcal S}^p$), так і просторів $B_2$-м.п. функцій (оскільки множини функцій $B_2$-м.п. збігаються з множинами $B{\mathcal S}^2$).

\subsection{Найкращі наближення в просторах $B {\mathcal S}^p$.}

Надалі, розглядатимемо лише ті майже періодичні функції з просторів $B{\mathcal S}^p$, послідовності показників Фур'є яких мають єдину граничну точку в нескінченності. Для таких функцій $f$,  ряди Фур'є можна записати в симетричній формі:
 \begin{equation}\label{Fourier_Series}
  S[f](x)=\sum _{k \in {\mathbb Z}}
  A_{k} {\mathrm e}^{ {\mathrm i}\lambda_k x},\quad  \mbox{\rm  де }\
  A_{k}=A_{k}(f)=M\{f(\cdot) {\mathrm e}^{-{\mathrm i}\lambda_k \cdot}\},
\end{equation}
$\lambda_0:=0$, $\lambda_{-k}=-\lambda_k$,  $|A_k|+|A_{-k}|>0$, $\lambda_{k+1}>\lambda_k>0$ при $k>0$.

Через $G_{\lambda_n}$  позначаємо множину всіх $B$-м.п. функцій $f$, показники Фур'є яких належать інтервалу
$(-\lambda_n,\lambda_n)$ і означаємо величину найкращого наближення функції $f$  рівністю
 \begin{equation}\label{Best_Approximation_BS}
 E_{\lambda_n}(f)_{_{\scriptstyle p}} = E_{\lambda_n}(f)_{_{\scriptstyle  B{\mathcal S}^p}}  =\inf\limits_{g\in G_{\lambda_n}}
 \|f-g \|_{_{\scriptstyle p}}.
\end{equation}

Зі співвідношень  (\ref{norm_Sp_BS}) та  (\ref{Best_Approximation_BS}) випливає, що для довільної   $f\in B{\mathcal S}^p$
з рядом Фур'є  (\ref{Fourier_Series}) маємо
  \begin{equation} \label{Best_Approx_BS}
     E_{\lambda_n}^p(f)_{_{\scriptstyle p}}=\|f-S_n(f)\|_{_{\scriptstyle p}} =\sum\limits_{|k|\ge n} |A_k(f)|^p,
 \end{equation}
де $S_n(f):=\sum _{|k|<n} A_{k}(f) {\mathrm e}^{ {\mathrm i}\lambda_k x}$.

\subsection{Узагальнені модулі гладкості в $B {\mathcal S}^p$}\label{Gen_MS_BSp}

Нехай $\Phi$ -- множина всіх неперервних обмежених невід'ємних парних функцій
$\varphi$ таких, що  $\varphi(0)=0$ і міра Лебега множини  $\{t\in {\mathbb R}:\,\varphi(t)=0\}$ дорівнює нулю.
Для довільної фіксованої $\varphi\in \Phi$ розглянемо узагальнений модуль гладкості функції  $f\in  B{\mathcal S}^p $
\begin{equation}\label{general_modulus_BS}
    \omega_\varphi(f,\delta)_{_{\scriptstyle   B{\mathcal S}^p}}  :=\sup\limits_{|h|\le \delta} \Big(\sum_{k\in {\mathbb Z}}
\varphi^p(\lambda_kh)|  A_k(f) |^p\Big)^{1/p},\quad  \delta\ge 0.
\end{equation}

Розглянемо зв'язок модулів (\ref{general_modulus_BS}) з деякими відомими модулями гладкості.
Нехай $\Theta=\{\theta_j\}_{j=0}^m$ -- довільний ненульовий набір комплексних чисел таких, що  $\sum_{j=0}^m \theta_j=0$.
Поставимо у відповідність  набору $\Theta$ різницевий оператор
$
    \Delta_h^{\Theta}(f)=\Delta_h^{\Theta}(f,t)=\sum_{j=0}^m \theta_j f(t-jh)
$
і модуль гладкості
 \[
    \omega_{\Theta}(f,\delta)_{_{\scriptstyle   B{\mathcal S}^p}} :=
    \sup\limits_{|h|\le \delta}\|\Delta_h^{\Theta} (f)\|_{_{\scriptstyle   B{\mathcal S}^p}}.
 \]
Зазначимо, що набір $\Theta(m)=\Big\{\theta_j=(-1)^j {m \choose j}, \ j=0,1,\ldots, m\Big\}$, $m\in {\mathbb N}$, є відповідником  модулю гладкості натурального порядку  $m$, тобто,
$
    \omega_{\Theta(m)}(f,\delta)_{_{\scriptstyle   B{\mathcal S}^p}}=
    \omega_m(f,\delta)_{_{\scriptstyle   B{\mathcal S}^p}}.
$
Неважко переконатися, що у випадку, коли  $\varphi_{\Theta}(t)= |\sum_{j=0}^m \theta_j {\mathrm e}^{-{\mathrm i} jt} |$, маємо
 $
 \omega_{\varphi_{\Theta}}(f,\delta)_{_{\scriptstyle   B{\mathcal S}^p}} =
 \omega_{\Theta}(f,\delta)_{_{\scriptstyle   B{\mathcal S}^p}}.
 $
Якщо ж   $\varphi_\alpha(t)= 2^\alpha   |\sin (t/2) |^\alpha = 2^{\frac {\alpha   }2} (1-\cos t)^{\frac \alpha2}$, $\alpha>0$, то
модуль $\omega_{\varphi_\alpha}(f,\delta)_{_{\scriptstyle   B{\mathcal S}^p}}$ є  модулем гладкості прядку $\alpha$:
 \[
     \omega_{\varphi_\alpha}(f,\delta)_{_{\scriptstyle   B{\mathcal S}^p}}=\omega_\alpha(f,t)_{_{\scriptstyle   B{\mathcal S}^p}} :=
    \sup\limits_{|h|\le t}\|\Delta_h^\alpha f\|_{_{\scriptstyle   B{\mathcal S}^p}} =
     \sup\limits_{|h|\le t} \Big\|\sum\limits_{j=0}^\infty (-1)^j {\alpha \choose j} f(\cdot-jh)
     \Big\|_{_{\scriptstyle   B{\mathcal S}^p}}.
 \]

Далі, нехай
$$
    F_h (f,t)=f_h(x)
    :=\frac{1}{2h} \int\limits_{t-h}^{t+h} f(u) du
$$
-- функція Стєклова функції $f \in {\mathcal S}^p.$ Розглянемо різниці
$$
 \widetilde{\Delta}_h^1 (f):= \widetilde{\Delta}_h^1 (f,t)=  F_h (f,t)-f(t)=(F_h-\mathbb{I}) (f,t),
$$
$$
  \widetilde{\Delta}_h^m (f):=\widetilde{\Delta}_h^m (f,t)=\widetilde{\Delta}_h^1(\Delta_h^{m-1}(f),t)=(F_h-\mathbb{I})^m (f,t)=
  \sum_{k=0}^m k^{m-k} {m \choose k} F_{h,k} (f,t),
$$
де $m=2,3,\ldots$, $F_{h,0}(f):=f,$ $F_{h,k}(f):=F_h(F_{h,k-1}(f))$ і $\mathbb{I}$ -- тотожній оператор в $B{\mathcal S}^p$.
Означимо такі гладкісні характеристики
\begin{equation}\label{tilde-omega-def}
    \widetilde{\omega}_m (f, \delta)_{_{\scriptstyle  B{\mathcal S}^p}}:= \sup_{0\le h \le \delta}
    \|\widetilde{\Delta}_h^m (f) \|_{_{\scriptstyle  B{\mathcal S}^p}}, \quad \delta>0.
\end{equation}
Виявляється, що  $\omega_{\tilde{\varphi}_m}(f,\delta)_{_{\scriptstyle  B{\mathcal S}^p}} =
 \widetilde{\omega}_m(f,\delta)_{_{\scriptstyle  B{\mathcal S}^p}}$ при $\tilde{\varphi}_m(t)=(1-\mathrm{sinc}~ t)^m,$ $m \in \mathbb{N},$ де  $\mathrm{sinc}~ t=\{(\sin t)/t, ~ \mbox{якщо} ~t\not=0, \quad 1,  ~\mbox{якщо} ~t=0\}.$

В загальному випадку модулі, подібні до  (\ref{general_modulus_BS}) розглядались в
\cite{Boman_1980, Vasil'ev_2001, Kozko_Rozhdestvenskii_2004, Vakarchuk_2016, Babenko_Savela_2012, Abdullayev_Chaichenko_Shidlich_2021, Abdullayev_Serdyuk_Shidlich_2021, Abdullayev_Chaichenko_Shidlich_2021_RMJ, Serdyuk_Shidlich_2022, Chaichenko_Shidlich_Shulyk_2022} та ін.



\subsection{Нерівності типу Джексона в просторах $B{\mathcal S}^p$}

У цьому підрозділі формулюються прямі теореми наближення майже періодичних функцій в просторах $B{\mathcal S}^p$
в термінах найкращих наближень та узагальнених модулів гладкості. Для функцій  $f\in B{\mathcal S}^p$ з рядом Фур'є вигляду  (\ref{Fourier_Series}), розглядаються нерівності типу Джексона вигляду
 $$
 E_{\lambda_n}(f)_{_{\scriptstyle  B{\mathcal S}^p}}\le
 K(\tau )\omega _\varphi\Big(f, \frac {\tau }{\lambda_n}\Big)_{_{\scriptstyle  B{\mathcal S}^p}}, \quad \tau >0, \quad 1\le p<\infty, \quad n\in {\mathbb N},
 $$
і досліджується питання щодо найкращих констант у цих нерівностях при фіксованих параметрах   $n$, $\varphi$, $\tau$ та $p$.
 Тобто,   досліджуються величини
 \[
 K_{n,\varphi,p}(\tau )=
  \sup \bigg \{\frac {E_{\lambda_n}(f)_{_{\scriptstyle  B{\mathcal S}^p}}}{\omega _\varphi(f, \frac {\tau}{\lambda_n})_{_{\scriptstyle  B{\mathcal S}^p}}}\ \ :
  f\in B{\mathcal S}^p\bigg\}.
 \]
Тут і далі ми вважаємо, що $0/0=0$.

Нехай $V(\tau)$, $\tau>0$, -- множина всіх обмежених неспадних відмінних від константи функцій  $v $ на $[0, \tau].$

 \begin{theorem}
   \label{Th.1} {\rm (\cite{Serdyuk_Shidlich_2022})}
   Нехай функція  $f\in B{\mathcal S}^p,$ $1\le p<\infty$, має ряд Фур'є вигляду \eqref{Fourier_Series}. Тоді для довільних  $\tau >0$, $n\in {\mathbb N}$ та $\varphi\in \Phi$ має місце нерівність
\begin{equation}\label{(6.7)}
 E_{\lambda_n}(f)_{_{\scriptstyle  B{\mathcal S}^p}}\le C_{n,\varphi,p}(\tau ) \omega _\varphi\Big(f, \frac {\tau }{\lambda_n}\Big)_{_{\scriptstyle  B{\mathcal S}^p}},
  \end{equation}
де
 \begin{equation}\label{(6.8)}
 C_{n,\varphi,p}(\tau ):=\left(\inf\limits _{v  \in  V(\tau )} \frac {v  (\tau )-v  (0)}
 {I_{n,\varphi,p}(\tau ,v  )}\right)^{1/p},
  \end{equation}
 і
 \begin{equation}\label{(6.9)}
  I_{n,\varphi,p}(\tau ,v  ):= \inf\limits _{{k \in {\mathbb N}},\,k \ge n}  \int\limits _0^{\tau }\varphi^p\Big(\frac{\lambda_k t}{\lambda_n} \Big){\mathrm d}v  (t).
 \end{equation}
 При цьому існує функція  $v _*\in V(\tau )$, яка реалізує точну нижню межу в \eqref{(6.8)}. Нерівність  \eqref{(6.7)} є точною
на множині всіх функцій $f\in B{\mathcal S}^p$ з рядом Фур'є вигляду \eqref{Fourier_Series} в тому сенсі, що для довільної $\varphi\in \Phi$ та будь-якого  $n\in {\mathbb N}$ виконується рівність
 \[
 C_{n,\varphi,p}(\tau )= K_{n,\varphi,p}(\tau ).
 \]
\end{theorem}


У просторах  $L_2$  $2\pi$-періодичних сумовних з квадратом функцій для модулів неперервності (модулів гладкості першого порядку), цей результат отримано
О.\,Г.~Бабенком \cite{Babenko_1986}. У просторах ${\mathcal S}^p$  функцій однієї та багатьох змінних аналогічні результати для класичних модулів гладкості
отримано відповідно в роботах  \cite{Stepanets_Serdyuk_UMZh2002} та   \cite{Abdullayev_Ozkartepe_Savchuk_Shidlich_2019},
а для узагальнених модулів гладкості -- в  \cite{Abdullayev_Chaichenko_Shidlich_2021} (для функцій однієї змінної). У роботі \cite{Chaichenko_Shidlich_Shulyk_2022}
подібне твердження доведено у просторах Безиковича-Мусєлака-Орлича. Воно збігається з твердженням теореми \ref{Th.1} при $p=1$.

Розглянемо важливий частинний випадок, коли
$\varphi(t)=\varphi_\alpha(t)=  2^{\frac {\alpha}2} (1-\cos t)^{\frac \alpha2}=2^\alpha   |\sin (t/2) |^\alpha $, $\alpha>0$.
В цьому випадку  $ \omega_{\varphi_\alpha}(f,\delta)_{_{\scriptstyle  B{\mathcal S}^p}}= \omega_\alpha(f,\delta)_{_{\scriptstyle  B{\mathcal S}^p}}$
і  $K_{n,\varphi_\alpha,p}(\tau )=:K_{n,\alpha,p}(\tau )$.
В наступному твердженні даються оцінки констант $K_{n, \alpha ,p}(\tau ) $ в нерівностях типу Джексона з модулями гладкості  $\omega _\alpha(f, \cdot)_{_{\scriptstyle  {p}}}$ при $\tau=\pi$. Ці оцінки не залежать від
$n$ і є непокращуваними в деяких важливих випадках.

 \begin{corollary}
   \label{Th.3} $(${\rm\cite{Serdyuk_Shidlich_2022}}$)$
   Для довільних  $n\in {\mathbb N}$ та $\alpha>0$ мають місце нерівності
  \begin{equation}\label{(6.14)}
 K_{n, \alpha ,p}^p(\pi )\le \frac 1{2^{\frac { \alpha  p }2-1}I_n(\frac { \alpha  p }2)}\le \frac {\frac { \alpha  p }2+1}{2^{ \alpha  p}+2^{\frac { \alpha  p }2-1}(\frac { \alpha  p}2+1)\sigma (\frac { \alpha  p }2)}, \end{equation}
 де
 \begin{equation}\label{(6.12)}
 I_n(s):=\inf\limits _{{k \in {\mathbb N}},\,k \ge n} \int\limits _0^{\pi } \Big(1-\cos \frac {\lambda_k t}{\lambda_n}\Big)^{s}\sin t\, {\mathrm d}t, \ \
 s >0, \ \ n\in {\mathbb N}.
   \end{equation}
 і
 \[
 \sigma ( s ):=-\sum _{\alpha=[\frac { s }2]+1}^{\infty } { s \choose 2\alpha}
\frac 1{2^{2\alpha-1}}\bigg(\frac {1-(-1)^{[ s ]}}2 {2\alpha \choose \alpha}
 - \sum _{j=0}^{\alpha-1} {2\alpha \choose j}\frac 2{2(\alpha-j)^2-1}\bigg),
 \]
$([ s ]$ -- ціла частина числа $ s )$. Якщо ж $\frac { \alpha  p }2\in {\mathbb N},$
то величина  $\sigma (\frac { \alpha  p }2)=0$ і
 $$
  K_{n, \alpha ,p}^p(\pi )\le \frac {\frac { \alpha  p }2+1}{2^{ \alpha  p}}, \ \ \
\frac { \alpha  p }2\in {\mathbb N}, \ \ \ n\in {\mathbb N}.
 \eqno (\ref{(6.14)}')
 $$
  \end{corollary}

Наступне твердження встановлює рівномірну обмеженість констант  $K_{n, \alpha ,p}(\pi )$
відносно параметрів   $n\in {\mathbb N}$ та $1\le p<\infty$.

\begin{corollary} \label{Theorem 2.4.}$(${\rm\cite{Serdyuk_Shidlich_2022}}$)$
   Нехай функція   $f\in B{\mathcal S}^p,$ $1\le p<\infty$, має ряд Фур'є вигляду \eqref{Fourier_Series}  і $\|f-A_0(f)\|_{_{\scriptstyle  p}}\not= 0$. Тоді для довільних   $n\in {\mathbb N}$ та $\alpha>0$
  \begin{equation}\label{(6.16)}
       E_{\lambda_n} (f)_{_{\scriptstyle  B{\mathcal S}^p}} < \frac {(4/3)^{1/p}}{2^{\alpha/2}}\omega_\alpha
  \Big(f, \frac {\pi }{\lambda_n}\Big)_{_{\scriptstyle  B{\mathcal S}^p}} \le  \frac {4}{3\cdot 2^{\alpha/2}}\omega_\alpha
  \Big(f, \frac {\pi }{\lambda_n}\Big)_{_{\scriptstyle  B{\mathcal S}^p}}.
  \end{equation}
При цьому якщо  $\alpha=m\in {\mathbb N}$, то
  \begin{equation}\label{(6.161)}
       E_{\lambda_n} (f)_{_{\scriptstyle  B{\mathcal S}^p}} < \frac {4-2\sqrt{2}}{2^{m/2}}\omega_m
  \Big(f, \frac {\pi }{\lambda_n}\Big)_{_{\scriptstyle  B{\mathcal S}^p}}.
   \end{equation}
\end{corollary}


 \begin{theorem}
       \label{Th.3.1.}
        Нехай функція  $f\in B{\mathcal S}^p,$ $1\le p<\infty$, має ряд Фур'є вигляду \eqref{Fourier_Series}, $\varphi\in \Phi$, $\tau>0$ і  ${v}\in  V(\tau )$.   Тоді для довільного    $n\in {{\mathbb N}}$ справджується нерівність

 \begin{equation}\label{Jackson_Type_Ineq_N}
        E_{\lambda_n}^p(f)_{_{\scriptstyle  B{\mathcal S}^p}}\le  \frac 1{I_{n,\varphi,p}(\tau ,v  )}
        \int\limits _0^{\tau }\omega _\varphi^p\Big(f, \frac t{\lambda_n}\Big)_{_{\scriptstyle  B{\mathcal S}^p}}{\mathrm d}v  (t),
 \end{equation}
де   величина $I_{n,\varphi,p}(\tau ,{v} )$ визначається рівністю \eqref{(6.9)}.
Якщо при цьому функція  $\varphi$ є неспадною на проміжку $[0,\tau]$, і виконується умова
 \begin{equation}\label{I_n,varphi,p_Equiv_nNN}
      I_{n,\varphi,p}(\tau ,{v} )=\int\limits _0^{\tau }\varphi^p(t) {\mathrm d} {v}   (t),
 \end{equation}
то нерівність  \eqref{Jackson_Type_Ineq_N} не може бути покращена і тому
 \begin{equation}
       \label{Jackson_Type_Exact_BSp}
       \mathop {\sup\limits _{f\in B{\mathcal S}^p}}\limits _{f\not ={\rm const }}
       \frac {E_{\lambda_n}(f)_{_{\scriptstyle  B{\mathcal S}^p}} }
       {\Big(\int\limits _0^{\tau }\omega _\varphi^p (f, \frac t{\lambda_n} )_{_{\scriptstyle  B{\mathcal S}^p}}{\mathrm d}v  (t)\Big)^{1/p} }=
       \Big(\int_0^{\tau }\varphi^p(t) {\mathrm d} {v}   (t)\Big)^{-1/p} .
 \end{equation}

\end{theorem}

\begin{proof} Нерівність (\ref{Jackson_Type_Ineq_N}) випливає зі співвідношення (12) роботи \cite{Serdyuk_Shidlich_2022}.
Якщо ж   функція  $\varphi$ є неспадною на $[0,\tau]$ і виконується співвідношення (\ref{I_n,varphi,p_Equiv_nNN}), то  на підставі  (\ref{Jackson_Type_Ineq_N}) маємо
 \begin{equation}\label{A6.94}
       \mathop {\sup\limits _{f\in B{\mathcal S}^p}}\limits _{f\not ={\rm const }}
       \frac {E_{\lambda_n}(f)_{_{\scriptstyle  B{\mathcal S}^p}} }
       {\Big(\int\limits _0^{\tau }\omega _\varphi^p (f, \frac t{\lambda_n} )_{_{\scriptstyle  B{\mathcal S}^p}}{\mathrm d}v  (t)\Big)^{1/p} }\le
       \Big(\int_0^{\tau }\varphi^p(t) {\mathrm d} {v}   (t)\Big)^{-1/p}
       .
   \end{equation}
Для доведення протилежної нерівності до  (\ref{A6.94}) розглянемо функцію
 \begin{equation}\label{f_exstremal}
  f_n(x)=\gamma  + \varepsilon   {\mathrm e}^{-{\mathrm i}\lambda_nx} +
    \varepsilon    {\mathrm e}^{{\mathrm i}\lambda_nx},
  \end{equation}
де  $\gamma$ та  $\varepsilon$ -- будь-які комплексні числа.

Оскільки функція  $\varphi(\lambda_nt)$ є неспадною на  $[0, \frac {\tau}{\lambda_n}]$, то внаслідок  (\ref{general_modulus_BS}) маємо
 \begin{equation}\label{A6.96}
\omega _\varphi (f_n, \delta)_{_{\scriptstyle  B{\mathcal S}^p}} =   2^{\frac 1p} |\varepsilon |\
  \varphi (\lambda_n\delta).
 \end{equation}
Враховуючи  (\ref{A6.96}) та рівність
 $
    E_{\lambda_n}
  (f_n)_{_{\scriptstyle  {\mathcal S}^p}}  =2^{\frac 1p} |\varepsilon |
  ,
 $
бачимо, що
  \[
  \mathop {\sup\limits _{f\in B{\mathcal S}^p}}\limits _{f\not ={\rm const}}
  \frac {E_{\lambda_n}  (f)_{_{\scriptstyle  B{\mathcal S}^p}} }
  {\Big(\int\limits _0^{\tau }\omega _\varphi^p (f, \frac t{\lambda_n} )_{_{\scriptstyle  B{\mathcal S}^p}}{\mathrm d}v  (t)\Big)^{1/p} }
  \ge
  \frac {E_{\lambda_n}  (f_n)_{_{\scriptstyle  B{\mathcal S}^p}} }
  {\Big(\int\limits _0^{\tau }\omega _\varphi^p (f_n, \frac t{\lambda_n} )_{_{\scriptstyle  B{\mathcal S}^p}}{\mathrm d}v  (t)\Big)^{1/p} }
  \]
  \begin{equation}\label{A6.97}
  = \frac {2^{\frac 1p} |\varepsilon |}
   {
    \Big(\int _0^{\tau /\lambda_n}2|\varepsilon |^p \varphi^p(\lambda_nt){\mathrm d} {v}  (\lambda_nt)\Big)^{1/p}
   }=
        \Big(\int_0^{\tau }\varphi^p(t) {\mathrm d} {v}   (t)\Big)^{-1/p}.
 \end{equation}
Із співвідношень (\ref{A6.94}) та (\ref{A6.97}) отримуємо  (\ref{Jackson_Type_Exact_BSp}).

 \end{proof}

 Зазначимо, що у просторах ${\mathcal S}^p$ аналогічне твердження буде випливати з наведеної нижче теореми \ref{Th.3.1.} при $\psi(k)\equiv 1$, $k\in {\mathbb Z}$.

Розглянемо   випадок, коли
$\varphi(t)=\varphi_\alpha(t)=  2^{\frac {\alpha}2} (1-\cos t)^{\frac \alpha2}=2^\alpha   |\sin (t/2) |^\alpha $, $\alpha>0$.
Для вагової функції  $v  _1(t) = 1 - \cos t$ отримуємо таке твердження:

 \begin{corollary}
   \label{Th.2}$(${\rm\cite{Serdyuk_Shidlich_2022}}$)$ Для довільної функції  $f\in B{\mathcal S}^p,$ $1\le p<\infty$,
 з рядом Фур'є вигляду \eqref{Fourier_Series} має місце нерівність
 \begin{equation}\label{(6.11)}
 E_{\lambda_n}^p (f)_{_{\scriptstyle  B{\mathcal S}^p}} \le \frac 1{2^{\frac {\alpha  p }2}I_n(\frac {\alpha p }2)}\int\limits_0^{\pi }
 \omega _\alpha^p\Big(f, \frac t{\lambda_n}\Big)_{_{\scriptstyle  B{\mathcal S}^p}}  \sin t\,{\mathrm d}t, \ \ n \in {\mathbb N},\ \alpha>0,
  \end{equation}
де величини $I_n(s)$, $ s >0$, визначаються в  \eqref{(6.12)}.
При цьому, якщо $\frac {\alpha p }2\in {\mathbb N}$, то
 \begin{equation}\label{(6.13)}
 I_n\Big(\frac { \alpha p }2\Big)=\frac {2^{\frac { \alpha  p }2+1}}{\frac { \alpha  p }2+1},
   \end{equation}
і нерівність  \eqref{(6.11)} не може бути покращеною ні при яких $n\in
{\mathbb N}.$
 \end{corollary}

Для вагової функції  $v  _2(t)=t$  маємо таке твердження:
 \begin{corollary}
   \label{Th.2a} $(${\rm\cite{Serdyuk_Shidlich_2022}}$)$ Нехай функція $f\in B{\mathcal S}^p,$ $1\le p<\infty$, має ряд Фур'є вигляду  \eqref{Fourier_Series}
   і число
   $\alpha>0$ таке, що  $\alpha p\ge 1$. Тоді для довільних   $0<\tau \le \frac {3\pi }4$ та $n \in {\mathbb N}$
  \begin{equation}\label{(A6.53)}
    E_{\lambda_n}^p (f)_{_{\scriptstyle  B{\mathcal S}^p}} \le \frac {1}{2^{\alpha p}\int _0^{\tau }\sin ^{\alpha p }\frac t{2}{\mathrm d}t} \int\limits_0^{\tau}
 \omega _\alpha^p\Big(f, \frac t{\lambda_n}\Big)_{_{\scriptstyle  B{\mathcal S}^p}} {\mathrm d}t .
      \end{equation}
Рівність в  \eqref{(A6.53)} справджується для функції $f_n$  вигляду  \eqref{f_exstremal}.
 \end{corollary}


Наслідки \ref{Th.3}, \ref{Theorem 2.4.}, \ref{Th.2} та \ref{Th.2a} отримано в роботі
\cite{Serdyuk_Shidlich_2022}. У просторах ${\mathcal S}^p$ періодичних функцій однієї та багатьох змінних аналогічні результати для класичних модулів гладкості
отримано в роботах  \cite{Stepanets_Serdyuk_UMZh2002, Voicexivskij_UMZh2003, Abdullayev_Ozkartepe_Savchuk_Shidlich_2019}.
У просторах $L_2$ для класичних модулів гладкості нерівність  \eqref{(6.11)} при $\alpha=1$ доведено М.\,І.~Чернихом \cite{Chernykh_1967_MZ}.

У випадку, коли  $\varphi(t)=\tilde{\varphi}_m(t)=(1-\mathrm{sinc}~ t)^m,$ $m \in \mathbb{N},$ (де, як і вище,
$\mathrm{sinc}\, t=\frac{\sin t}{t}$ при $t\not=0$ і $\mathrm{sinc}\, t=1$ при $t=0$),  $\tau=\pi$ і  $v(u)=1-\cos u,$ $u\in [0,\pi],$ із співвідношень \eqref{(6.7)}--\eqref{(6.9)}  отримуємо
 $$
    E_{\lambda_n}^p(f)_{_{\scriptstyle  B{\mathcal S}^p}}\le
    \frac {1}{\tilde {I}_{n}(m)} \int\limits_0^{\pi }
    \tilde{\omega}_m^p \Big(f,\frac{u}{\lambda_n}\Big)_{_{\scriptstyle  B{\mathcal S}^p}}
    \sin u \,{\mathrm d} u,
$$
де
$$
    \tilde{I}_{n}(m)= \inf_{k \in \mathbb{N}, k\ge n} \int_0^\pi
    \Big(1-\mathrm{sinc} ~\frac{\lambda_k u}{\lambda_n}\Big)^{pm}
    \sin u \,{\mathrm d} u.
$$

Оскільки
$$
    1- \mathrm{sinc}~\Big(\frac{\lambda_k u}{\lambda_n}\Big) \ge 1-\frac{\sin u}{u}\ge \Big(\frac{u}{\pi}\Big)^2, \quad k\ge n, \quad u \in [0,\pi],
$$
(див., наприклад, \cite{Vakarchuk+Zabutnaya-Math-notes-2016}), то
$$
    \tilde{I}_{n}(m)\ge \int_0^\pi
    \Big(1-\mathrm{sinc} ~ u \Big)^{pm}   \sin u \,{\mathrm d} u \ge
    \frac{1}{\pi^{2pm}}\int_0^\pi u^{2pm} \sin u \,{\mathrm d} u.
$$
Якщо ж при цьому $p\in {\mathbb N}$, то
\begin{equation} \label{K_m}
    \tilde{I}_{n}(m)\ge \frac{2(pm)!}{\pi^{2pm}} \Big(\sum_{j=0}^{pm} (-1)^j \frac{\pi^{2pm-2j}}{(2pm-2j)!}+
    \frac{\pi^{2pm}}{2(pm)!}(-1)^{pm} \Big)=:\frac{2(pm)!}{\pi^{2pm}} {\mathscr K}(pm).
\end{equation}
Отже, має місце таке твердження.


 \begin{corollary}\label{Cor.3}
    Нехай  $f\in B{\mathcal S}^p,$ $p\in {\mathbb N}$, -- довільна функція
    з рядом Фур'є вигляду \eqref{Fourier_Series}. Тоді для довільних натуральних $n$ та  $m>0$
\begin{equation} \label{cor-inequal-Omega-m}
    E_{\lambda_n}^p (f)_{_{\scriptstyle  B{\mathcal S}^p}} \le
    \frac {\pi^{2pm}}{2(pm)! \  {\mathscr K}(pm)} \int\limits_0^{\pi }
    \tilde{\omega}_m^p \Big(f,\frac{u}{\lambda_n}\Big)_{_{\scriptstyle  B{\mathcal S}^p}}
    \sin u \, {\mathrm d} u,
\end{equation}
де  ${\mathscr K}(pm)$ визначається в співвідношенні (\ref{K_m}).
\end{corollary}

У випадку, коли $p=1$, дане твердження випливає з наслідку 3.4 роботи \cite{Chaichenko_Shidlich_Shulyk_2022}. У просторах $L_2$
$2\pi$-періодичних сумовних з квадратом функцій аналогічні результати отримано в роботах \cite{Abilov+Abilova-Math-note-2004, Vakarchuk_2016}.

Як зазначалось  вище, при $p=2$ множини $B{\mathcal S}^p$ збігаються з множинами $B_2$-м.п. функцій. З огляду на важливість цього випадку
наведемо відповідні твердження для класичних модулів гладкості $\omega_m $, які випливають з наслідку  \ref{Th.2}.

 \begin{corollary} \label{Theorem 2.5.}
     Для довільної $B_2$-м.п. функції $f$ з рядом Фур'є вигляду \eqref{Fourier_Series},
  \begin{equation}\label{(6.162)}
     E_{\lambda_n}^2 (f)_{_{\scriptstyle  {2}}} \le \frac {m+1}{2^{2m+1}}
     \int\limits_0^{\pi } \omega _m^2\Big(f, \frac t{\lambda_n}\Big)_{_{\scriptstyle  {2}}}  \sin t\,{\mathrm d}t,\quad m,\, n\in {\mathbb N}.
   \end{equation}
Нерівність (\ref{(6.162)}) не може бути покращеною ні при яких  $m$ та $n\in {\mathbb N}$.
\end{corollary}

 \begin{corollary} \label{Theorem 2.6.}
     Для довільної $B_2$-м.п. функції $f$ з рядом Фур'є вигляду \eqref{Fourier_Series} такої, що
     $\|f-A_0(f)\|_{_{\scriptstyle  p}}\not= 0$, має місце нерівність
     \begin{equation}\label{(6.163)}
       E_{\lambda_n} (f)_{_{\scriptstyle  {2}}} < \frac {\sqrt{m+1}}{2^m}\omega_m
       \Big(f, \frac {\pi }{\lambda_n}\Big)_{_{\scriptstyle 2}},\quad m,\, n\in {\mathbb N}.
     \end{equation}
 \end{corollary}


Нерівності  (\ref{(6.162)}) та (\ref{(6.163)}) доповнюють результати, отримані в  \cite{Prytula_1972, Prytula_Yatsymirskyi_1983, Babenko_Konareva_2019} та ін., для
$B_2$-м.п. функцій.  Нерівності подібного типу також досліджувалися в
\cite{Taikov_1976, Taikov_1979, Vasil'ev_2001, , Serdyuk_2003, Vakarchuk_2004, Babenko_Savela_2012, Vakarchuk_2016, Abdullayev_Serdyuk_Shidlich_2021} та ін.

Для просторів $L_2$ нерівності \eqref{(6.162)} та \eqref{(6.163)} встановлено в роботах М.\,І.~Черниха
\cite{Chernykh_1967, Chernykh_1967_MZ} при $m=1$ та у роботі О.\,І.~Степанця та А.\,С.~Сердюка \cite{Stepanets_Serdyuk_UMZh2002} і роботі С.\,М.~Васильєва \cite{Vasil'ev_2001} при $n\in {\mathbb N}$. Зазначені роботи започаткували у даній проблематиці цілий напрям досліджень, який активно розвивався в роботах В.Р.~Войцехівського \cite{Voicexivskij_2002, Voicexivskij_2003, Voicexivskij_UMZh2003}, С.\,Б.~Вакарчука \cite{Vakarchuk_2004, Vakarchuk_Shchitov_2006}, М.\,Л.~Горбачука, Я.\,I.~Грушки, С.\,М.~Торби \cite{Gorbachuk_Grushka_Torba_2005} та ін.




\subsection{Обернені теореми наближення в просторах $B{\mathcal S}^p$}

 \begin{theorem}
       \label{Inverse_Theorem} $(${\rm\cite{Serdyuk_Shidlich_2022}}$)$
       Нехай функція  $ f\in  B{\mathcal S}^p,$ $1\le p<\infty$, має ряд Фур'є вигляду  \eqref{Fourier_Series},
       функція  $\varphi\in \Phi$ не спадає на   $[0,\tau]$, $\tau>0$, і
       $\varphi(\tau)=\max\{\varphi(t):t\in {\mathbb R}\}$. Тоді для довільного    $n\in {\mathbb N}$ має місце нерівність
       \begin{equation}\label{S_M.12}
       \omega _\varphi^p\Big(f, \frac{\tau}{\lambda_n}\Big)_{_{\scriptstyle  B{\mathcal S}^p}}
       \le    \sum _{\nu =1}^{n}\Big(\varphi^p\Big(\frac{\tau\lambda_\nu }{\lambda_n}\Big)-
          \varphi^p\Big(\frac{ \tau\lambda_{\nu-1}}{\lambda_n}\Big)\Big) E_{\lambda_\nu}^p(f)_{_{\scriptstyle  B{\mathcal S}^p}}
       .
       \end{equation}
 \end{theorem}

Функція $\varphi(t)=\varphi_\alpha  (t)=2^\alpha   |\sin (t/2) |^\alpha  $, $\alpha  >0$, задовольняє умови теореми
 \ref{Inverse_Theorem} з $\tau=\pi$.  Якщо  $r=\alpha p>0$, то використовуючи нерівності
$x^r  -y^r   \le r   x^{r  -1}(x-y)$ при $r\ge 1$  та  $x^r  -y^r   \le r   y^{r  -1}(x-y)$ при $0<r  <1$,  $x>0, y>0$ (див., наприклад, \cite[Гл.~1]{Hardy_Littlewood_Polya_1934}),
для всіх $\nu=1,2,\ldots,n$  отримуємо
\[
  \varphi^p\Big(\frac {\tau \lambda_\nu}{\lambda_n}\Big)-
  \varphi^p\Big(\frac {\tau \lambda_{\nu-1}}{\lambda_n}\Big)=
  2^{\alpha p}   \Big(\Big|\sin  \frac {\pi \lambda_\nu}{\lambda_n}  \Big|^{\alpha p}  -
  \Big|\sin  \frac {\pi \lambda_{\nu-1}}{\lambda_n}  \Big|^{\alpha p}  \Big)
\]
\[
  \le {\alpha p}   \Big(\frac{2\pi }{\lambda_n}\Big)^{\alpha p}   \lambda_\nu^{{\alpha p}  -1}(\lambda_{\nu}-\lambda_{\nu-1}).
\]

 \begin{corollary}
       \label{Corollary 21}$(${\rm\cite{Serdyuk_Shidlich_2022}}$)$
       Нехай функція $ f\in  B{\mathcal S}^p$,    $1\le p<\infty$,  має ряд Фур'є вигляду
       \eqref{Fourier_Series}. Тоді для довільних    $n\in {\mathbb N}$ та  $\alpha>0$
       \begin{equation}\label{Inverse_Inequality}
       \omega _\alpha^p \Big(f, \frac{\pi}{\lambda_n}\Big)_{_{\scriptstyle  B{\mathcal S}^p}}
       \le      \alpha p \Big(\frac{2\pi }{\lambda_n}\Big)^
       {\alpha p}
       \sum _{\nu =1}^{n}  \lambda_\nu^{\alpha p-1}(\lambda_{\nu}-\lambda_{\nu-1}) E_{\lambda_\nu}^p (f)_{_{\scriptstyle  B{\mathcal S}^p}}.
       \end{equation}
       Якщо при цьому показники Фур'є  $\lambda_\nu$, $\nu\in {\mathbb N}$, задовольняють умову
       \begin{equation}\label{Lambda_Cond}
        \lambda_{\nu+1}-\lambda_{\nu} \le K ,\quad \nu=1,2,\ldots,
       \end{equation}
      з деякою абсолютною сталою  $K>0$, то
       \begin{equation}\label{Inverse_Inequality_for_using}
       \omega _\alpha^p \Big(f, \frac{\pi}{\lambda_n}\Big)_{_{\scriptstyle  B{\mathcal S}^p}}\le      \frac{\alpha p (2\pi)^{\alpha p}}{\lambda_n^{\alpha p}} K
       \sum _{\nu =1}^{n}  \lambda_\nu^{{\alpha p}-1}  E_{\lambda_\nu}^p (f)_{_{\scriptstyle  B{\mathcal S}^p}}.
       \end{equation}
\end{corollary}

Далі ми покажемо, що нерівності  \eqref{Inverse_Inequality} та \eqref{Inverse_Inequality_for_using} можна покращити, а саме константу $2^{\alpha p}$, в  їх правих частинах можна замінити на 1.

\begin{theorem}
       \label{Inverse_Theorem_2}
       Нехай функція $ f\in  B{\mathcal S}^p$,    $1\le p<\infty$,  має ряд Фур'є вигляду
       \eqref{Fourier_Series}. Тоді для довільних    $n\in {\mathbb N}$ та  $\alpha>0$ має місце нерівність
       \begin{equation}\label{S_M.12}
       \omega _\alpha^p\Big(f, \frac{\tau}{\lambda_n}\Big)_{_{\scriptstyle  B{\mathcal S}^p}}
       \le    \Big({\frac {\pi  }{\lambda_n}}\Big)^{\alpha p}
     \sum _{\nu =1}^{n}(\lambda_\nu^{\alpha p}-\lambda_{\nu-1}^{\alpha p})
     E^p_{\lambda_{\nu}}  (f)_{_{\scriptstyle B{\mathcal S}^p}} .
       \end{equation}
 \end{theorem}

\begin{proof}
Використаємо схему доведення з робіт  \cite{Stepanets_Serdyuk_UMZh2002, Abdullayev_Chaichenko_Shidlich_2021},
з урахуванням особливостей просторів  $B{\mathcal S}^p$. Для довільної функції $f\in B{\mathcal S}^p$ та будь-якого $h\in {\mathbb R}$ маємо
\begin{equation}\label{Proof_Th2_1}
   \|\Delta_h^\alpha f\|_{_{\scriptstyle  B{\mathcal S}^p}}^p=
  \sum_{{k}\in {\mathbb Z}} 2^{\alpha p}   \Big|\sin \frac{\lambda_kh}{2} \Big|^{\alpha p}|A_k(f)|^p
  $$
  $$
  =\sum_{|k|< n} 2^{\alpha p}   \Big|\sin \frac{\lambda_kh}{2} \Big|^{\alpha p}|A_k(f)|^p+
   \sum_{|k|\ge n} 2^{\alpha p}   \Big|\sin \frac{\lambda_kh}{2} \Big|^{\alpha p}|A_k(f)|^p
  .
   \end{equation}
Оскільки другий доданок в  (\ref{Proof_Th2_1}) не перевищує величини
  \[
  2^{\alpha p} \sum_{|k|\ge n} |A_k(f)|^p= 2^{\alpha p}E_{\lambda_n}^p(f)_{_{\scriptstyle  B{\mathcal S}^p}},
   \]
і при  $h\in [0, \pi /\lambda_n]$ маємо
 \[
 \sum_{|k|< n} 2^{\alpha p}   \Big|\sin \frac{\lambda_kh}{2} \Big|^{\alpha p}|A_k(f)|^p\le
 \sum_{|k|< n} |\lambda_kh|^{\alpha p}|A_k(f)|^p
 \]
 \[
 \le  \Big({\frac {\pi  }{\lambda_n}}\Big)^{\alpha p}
 \sum_{|k|< n} |\lambda_k|^{\alpha p}|A_k(f)|^p=
 \Big({\frac {\pi  }{\lambda_n}}\Big)^{\alpha p}
 \sum_{\nu=1}^{n-1} \lambda_\nu^{\alpha p}H_\nu^p(f),
 \]
 де $H_\nu(f)=\Big(|A_{-\nu}(f)|^p+|A_\nu(f)|^p\Big)^{1/p}$, $\nu=1,2,\ldots$, то
\begin{equation}\label{(6.73)}
   \|\Delta_h^\alpha f\|_{_{\scriptstyle  B{\mathcal S}^p}}^p\le
      \Big({\frac {\pi  }{\lambda_n}}\Big)^{\alpha p}
 \sum_{\nu=1}^{n-1} \lambda_\nu^{\alpha p}H_\nu^p(f)+2^{\alpha p}E_{\lambda_n}^p(f)_{_{\scriptstyle  B{\mathcal S}^p}}.
    \end{equation}
Далі, використаємо наступне твердження з \cite{Stepanets_Serdyuk_UMZh2002}.

\begin{lemma}\label{Lemma_3} {\bf \cite{Stepanets_Serdyuk_UMZh2002}.}  Нехай числовий ряд  $\sum _{\nu =1}^{\infty} c_{\nu }$ є збіжним. Тоді для довільної послідовності  $a_{\nu },$ $\nu \in {\mathbb N},$ та будь-яких натуральних чисел  ${N_1}$  та ${N_2},$ ${N_1}\le {N_2}:$
\begin{equation}\label{(6.74)}
 \sum _{\nu ={N_1}}^{N_2}a_{\nu }c_{\nu }=a_{N_1}\sum _{\nu={N_1}}^{\infty }c_{\nu }+\sum _{\nu ={N_1}+1}^{N_2}(a_{\nu } -a _{\nu-1})\sum _{i=\nu }^{\infty }c_i-a_{N_2}\sum _{\nu ={N_2}+1}^{\infty }c_{\nu}.
    \end{equation}
\end{lemma}

Покладемо  в умовах леми \ref{Lemma_3} $a_{\nu }= \lambda_\nu^{\alpha p}$, $c_{\nu }=H_\nu^p(f)$, ${N_1}=1$ і  ${N_2}=n-1$. Згідно з   (\ref{(6.74)}) маємо
    \begin{equation}\label{(6.76aa)}
 \sum_{\nu=1}^{n-1} \lambda_\nu^{\alpha p}H_\nu^p(f)=\lambda_1^{\alpha p}
 \sum _{\nu =1}^{\infty}H_{\nu}^p(f)+ \sum _{\nu =2}^{n-1}(\lambda_\nu^{\alpha p}-\lambda_{\nu-1}^{\alpha p})
 \sum _{i=\nu}^{\infty}H_{i}^p(f)-\lambda_{n-1}^{\alpha p} \sum _{\nu =n}^{\infty}H_{\nu}^p(f).
    \end{equation}
Із \eqref{(6.73)}, \eqref{(6.76aa)} та \eqref{Best_Approx_BS} отримуємо оцінку
 \[
  \|\Delta _h^\alpha f\|_{_{\scriptstyle B{\mathcal S}^p}}^p\le
      \Big({\frac {\pi  }{\lambda_n}}\Big)^{\alpha p}
     \bigg( \sum _{\nu =1}^{n-1}(\lambda_\nu^{\alpha p}-\lambda_{\nu-1}^{\alpha p})
     E^p_{\lambda_{\nu}}  (f)_{_{\scriptstyle B{\mathcal S}^p}}-\lambda_{n-1}^{\alpha p }{E^p_{\lambda_n}} (f)_{_{\scriptstyle B{\mathcal S}^p}}\bigg)
 \]
    \begin{equation}\label{(6.76)}
 +2^{\alpha p}E_{\lambda_n}^p(f)_{_{\scriptstyle  B{\mathcal S}^p}}\le   \Big({\frac {\pi  }{\lambda_n}}\Big)^{\alpha p}
     \sum _{\nu =1}^{n}(\lambda_\nu^{\alpha p}-\lambda_{\nu-1}^{\alpha p})
     E^p_{\lambda_{\nu}}  (f)_{_{\scriptstyle B{\mathcal S}^p}},
    \end{equation}
з якої випливає (\ref{S_M.12}).

\end{proof}

Зазначимо, що  стала $\pi^\alpha$ в (\ref{S_M.12}) є точною в тому сенсі, що для довільного $\varepsilon>0$ існує функція
$f^*\in B{\mathcal S}^p$ така, що при всіх  $n$, більших певного числа $n_0$, маємо
 \begin{equation}\label{(6.31abcd)}
  \omega_{\alpha}^{\ }\Big(f^*, \frac {\pi }{\lambda_n}\Big)_{_{\scriptstyle B{\mathcal S}^p}}>
 \frac {\pi ^\alpha-\varepsilon }{\lambda_n^\alpha }\Big(\sum _{\nu =1}^{n}(\lambda_\nu^{\alpha p}-\lambda_{\nu-1}^{\alpha p})
     E^p_{\lambda_{\nu}}  (f^*)_{_{\scriptstyle B{\mathcal S}^p}}\Big)^{1/p}.
 \end{equation}
 Розглянемо функцію $f^*(x)={\rm e}^{\mathrm{i}\lambda_{k_0}x}$, де $k_0$ -- довільне натуральне число. Тоді $E_{\lambda_\nu}(f^*)_{_{\scriptstyle B{\mathcal S}^p}}=1$ при $\nu=1,2,\ldots,k_0$,
 $E_{\lambda_\nu}(f^*)_{_{\scriptstyle B{\mathcal S}^p}}=0$  при  $\nu>k_0$ і
 \[
  \omega_{\alpha}^{\ }\Big(f^*, \frac {\pi }{\lambda_n}\Big)_{_{\scriptstyle B{\mathcal S}^p}}\ge
 \|\Delta _{\frac {\pi }{\lambda_n}}^\alpha f^*\|_{_{\scriptstyle B{\mathcal S}^p}} \ge
  2^\alpha\Big|\sin \frac {\lambda_{k_0} \pi}{2{\lambda_n}}\Big|^{\alpha  }.
 \]
 Оскільки  ${(\sin t)}/t$ прямує до  $1$ при  $t\to 0$, то для всіх $n$, більших деякого числа $n_0$, виконується нерівність   $2^\alpha |\sin ({\lambda_{k_0} \pi}/{(2{\lambda_n}))}|^\alpha>
 {(\pi^\alpha-\varepsilon)} \lambda_{k_0}^\alpha/{{\lambda_n}^\alpha}$, з якої випливає  (\ref{(6.31abcd)}).


 \begin{corollary}
       \label{Corollary 22}
       Нехай функція $ f\in  B{\mathcal S}^p$,    $1\le p<\infty$,  має ряд Фур'є вигляду
       \eqref{Fourier_Series}. Тоді для довільних    $n\in {\mathbb N}$ та  $\alpha>0$
\begin{equation}\label{Inverse_Inequality11}
       \omega _\alpha^p \Big(f, \frac{\pi}{\lambda_n}\Big)_{_{\scriptstyle  B{\mathcal S}^p}}
       \le      \alpha p \Big(\frac{\pi }{\lambda_n}\Big)^       {\alpha p}
       \sum _{\nu =1}^{n}  \lambda_\nu^{\alpha p-1}(\lambda_{\nu}-\lambda_{\nu-1}) E_{\lambda_\nu}^p (f)_{_{\scriptstyle  B{\mathcal S}^p}}.
       \end{equation}
       Якщо $\lambda_\nu$, $\nu\in {\mathbb N}$, задовольняють умову (\ref{Lambda_Cond})
      зі сталою  $K>0$, то
       \begin{equation}\label{Inverse_Inequality_for_using11}
       \omega _\alpha^p \Big(f, \frac{\pi}{\lambda_n}\Big)_{_{\scriptstyle  B{\mathcal S}^p}}\le
       K\alpha p \Big(\frac{\pi }{\lambda_n}\Big)^       {\alpha p}
       \sum _{\nu =1}^{n}  \lambda_\nu^{{\alpha p}-1}  E_{\lambda_\nu}^p (f)_{_{\scriptstyle  B{\mathcal S}^p}}.
       \end{equation}
\end{corollary}


\subsection{Конструктивні характеристики функціональних класів, які визначаються модулями гладкості в
 $ B{\mathcal S}^{p}$}

Нехай  $\omega$ -- деяка функція (мажоранта) визначена на  $[0,1]$. При фіксованому  $\alpha>0$ позначимо
\begin{equation} \label{omega-class}
    B{\mathcal S}^{p} H^{\omega}_{\alpha} :=
    \Big\{f\in B{\mathcal S}^{p} :  \quad \omega_\alpha(f, \delta)_{_{\scriptstyle  B{\mathcal S}^p}}=
    {\mathcal O}  (\omega(\delta)),\quad  \delta\to 0+\Big\}.
\end{equation}
Далі, будемо розглядати мажоранти
$\omega(\delta)$, $\delta\in [0,1]$, які задовольняють наступні умови 1)--4): \noindent  {1)} $\omega(\delta)$ є неперервною на $[0,1]$;\
 {  2)} $\omega(\delta)\uparrow$;\   {  3)}
$\omega(\delta)\not=0$ при  $\delta\in (0,1]$;\   {  4)}~$\omega(\delta)\to 0$  при $\delta\to 0$; а також додаткову умову
\begin{equation} \label{B_alpha}
\quad \sum_{\nu=1}^{n} \lambda_{\nu}^{s-1}\omega (\lambda_{\nu}^{-1} ) =
{\mathcal O}  \Big(\lambda_n^s \omega  ( \lambda_n^{-1} )\Big).
\end{equation}
де $s>0$, а  $\lambda_\nu$, $\nu\in {\mathbb N}$, -- довільна зростаюча послідовність додатних чисел.
У випадку, коли $\lambda_\nu= \nu$, умова  $(\ref{B_alpha})$ є відомою умовою Барі
$({\mathscr B}_s)$ (див., наприклад, \cite{Bari_Stechkin_1956}).

 \begin{theorem}\label{Theorem 6.1} $(${\rm\cite{Serdyuk_Shidlich_2022}}$)$  Нехай функція $f\in B{\mathcal S}^{p}$,  $1\le p<\infty$,
  має ряд Фур'є вигляду  \eqref{Fourier_Series},    $\alpha>0$ і мажоранта  $\omega $ задовольняє умови
       $1)$--\,$4)$.

       i) Якщо  $f\in B{\mathcal S}^{p}H^{\omega}_{\alpha}$, то має місце співвідношення
     \begin{equation} \label{iff-theorem}
         E_{\lambda_n}(f)_{_{\scriptstyle  B{\mathcal S}^p}}={\mathcal O}
         \Big(\omega  ( \lambda_n^{-1} )\Big) .
      \end{equation}

      ii) Якщо ж числа $\lambda_\nu$, $\nu\in {\mathbb N}$ задовольняють умову \eqref{Lambda_Cond}, а функція
       $\omega^p $ -- умову  $(\ref{B_alpha})$ з  $s=\alpha p$, то виконання співвідношення
       \eqref{iff-theorem} забезпечує вкладення $f\in B{\mathcal S}^{p}H^{\omega}_{\alpha}$.
\end{theorem}

Функція  $t^r$, $0<r\le \alpha,$ задовольняє умову  (\ref{B_alpha}). Тому позначивши через
 $B{\mathcal S}^{p}H_{\alpha}^r$ множину  $B{\mathcal S}^{p}H^{\omega}_{\alpha}$ при
  $\omega(t)=t^r$, отримаємо таке твердження:

\begin{corollary}\label{corollary 6.1.}$(${\rm\cite{Serdyuk_Shidlich_2022}}$)$  Нехай $f \in B{\mathcal S}^{p}$,  $1\le p<\infty$,  має ряд Фур'є вигляду  \eqref{Fourier_Series}, $\alpha >0$, $0<r\le \alpha/p$ і виконується умова  \eqref{Lambda_Cond}.
 Функція $f$ належить множині  $B{\mathcal S}^{p}H_{\alpha}^r$ тоді і лише тоді, коли має місце співвідношення
$$
    E_{\lambda_n}(f)_{_{\scriptstyle  B{\mathcal S}^p}}={\mathcal O}   ({\lambda_n^{-r}} ).
$$
\end{corollary}

 У просторах $L_p$
 $2\pi$-періодичних сумовних за Лебегом в степені $p$ функцій, нерівності вигляду (\ref{Inverse_Inequality_for_using})
 отримано М.\,П.~Тіманом  (див., наприклад,  \cite[Гл.~6]{A_Timan_M1960}, \cite[Гл.~2]{Timan_M2009}).

В просторах  ${\mathcal S}^p$ для класичних модулів гладкості  $\omega_m$, теореми \ref{Inverse_Theorem}
та \ref{Theorem 6.1} доведено в \cite{Stepanets_Serdyuk_UMZh2002, Abdullayev_Ozkartepe_Savchuk_Shidlich_2019}, а нерівності вигляду
 (\ref{Inverse_Inequality_for_using11}) також отримані в  \cite{Sterlin_1972}. У роботах \cite{Chaichenko_Shidlich_Abdullayev_2019, Abdullayev_Chaichenko_Imash_kyzy_Shidlich_2020, Abdullayev_Chaichenko_Shidlich_2021}
отримано аналоги цих тверджень для подібних просторів типу Орлича, а у роботі \cite{Chaichenko_Shidlich_Shulyk_2022} -- у просторах
Безиковича-Мусєлака-Орлича. Зазначимо, що у випадку, коли $p=1$  теорема \ref{Inverse_Theorem_2} та наслідок \ref{Corollary 22}
випливають із результатів роботи \cite{Chaichenko_Shidlich_Shulyk_2022}.


\section{Прямі та обернені теореми наближення і поперечники функціональних класів у просторах ${\mathcal S}^p$}

\subsection{Узагальнені модулі гладкості та їх усереднені значення}
Повернемося до розгляду просторів ${\mathcal S}^p({\mathbb T}^d)$. Обмежимось випадком   $d=1$. Через ${\mathcal S}^p:={\mathcal S}^p({\mathbb T}^1)$,  $ 1\le p<\infty$, позначимо простір
визначених на дійсній осі  $2\pi$-періодичних комплексно значних сумовних за Лебегом на $[0,2\pi]$ функцій
$f$  $(f\in L({\mathbb T}^1))$, зі скінченною  нормою
 \begin{equation}\label{norm_Sp}
 \|f\|_{_{\scriptstyle  {\mathcal S}^p}} := \|\{\widehat f(k)\}_{k\in {\mathbb Z}}\|_{_{\scriptstyle  l_p({\mathbb Z})}} = \Big(\sum _{k\in {\mathbb Z}}|\widehat f({k})|^p\Big)^{1/p},
\end{equation}
де
 $\widehat f(k)=\frac 1{2\pi}\int_{0}^{2\pi }f(x){\mathrm e}^{-{\mathrm i}k x}{\mathrm d}x$ -- коефіцієнти Фур'є функції $f$.

Нехай, як і в п.\ref{Gen_MS_BSp},  для довільної фіксованої функції
$\varphi\in \Phi$  узагальнений модуль гладкості функції  $f\in {\mathcal S}^p $  задається рівністю
\[
    \omega_\varphi(f,\delta)_{_{\scriptstyle  {\mathcal S}^p}}
  :=\sup\limits_{|h|\le \delta}  \|\{\varphi(kh)\widehat f(k)\}_{k\in {\mathbb Z}}\|_{_{\scriptstyle  l_p({\mathbb Z})}} =
  \sup\limits_{|h|\le \delta} \Big(\sum_{k\in {\mathbb Z}}
\varphi^p(kh) | \widehat{f}(k) |^p\Big)^{1/p},\quad  \delta\ge 0.
\]

Далі, нехай   $V(\tau )$, $\tau>0$, --  множина всіх функцій  ${v}$, обмежених неспадних та відмінних від сталої на відрізку    $[0, \tau]$. Через  $\Omega _\varphi(f, \tau, {v}, u)_{_{\scriptstyle  {\mathcal S}^p}} $, $u>0$, позначимо усереднене значення узагальненого модуля гладкості $\omega _\varphi$ функції  $f$ з вагою  ${v} \in V(\tau )$:
 \begin{equation}\label{Mean_Value_Gen_Modulus}
  \Omega _\varphi(f, \tau, {v}, u)_{_{\scriptstyle  {\mathcal S}^p}} :=\bigg (\frac
   {1}{{v}  (\tau ) - {v}  (0)}\int 
   _0^u\omega _\varphi^p(f, t)_{_{\scriptstyle  {\mathcal S}^p}} {\mathrm d}{v}  \Big(\frac {\tau
   t}{u}\Big)\bigg)^{1/p}.
 \end{equation}
 Через $\Omega _\alpha(f, \tau, {v}, u)_{_{\scriptstyle  {\mathcal S}^p}} $  позначимо усереднене значення класичного модуля гладкості
порядку   $\alpha$ функції $f$ з вагою  ${v}\in V(\tau )$, тобто,
 $
  \Omega _\alpha(f, \tau, {v}, u)_{_{\scriptstyle  {\mathcal S}^p}} :=\Omega _\varphi(f, \tau, {v}, u)_{_{\scriptstyle  {\mathcal S}^p}}
 $   при  $\varphi(t)=\varphi_\alpha(t)=
  2^\frac \alpha 2 (1-\cos{t})^\frac \alpha2.
   $%

Зазначимо, що для довільних   $f\in {\mathcal S}^p,$ $\tau >0,$ ${v}
\in V(\tau )$ та $u>0$ функціонали  $\Omega _\varphi(f, \tau, {v}, u)_{_{\scriptstyle  {\mathcal S}^p}} $
не перевищують значення  $\omega _\varphi(f,u)_{_{\scriptstyle  {\mathcal S}^p}}$:
 \begin{equation}\label{general_modulus_INEQ}
    \Omega _\varphi(f, \tau, {v}, u)_{_{\scriptstyle  {\mathcal S}^p}} \le \omega _\varphi(f,u)_{_{\scriptstyle  {\mathcal S}^p}},
\end{equation}
 і тому в низці випадків вони можуть бути більш ефективними для характеризації структурних та апроксимативних властивостей функції  $f$.

У випадку, коли    $p=2$, $\psi (k)=k^{-r}$, $r\in \mathbb{N}$ і   $\mu(t)=t$,  класи, подібні до класів  $L^{\psi }(\alpha,\tau ,\mu , n)_{_{\scriptstyle  {\mathcal S}^p}}$ та $L^{\psi }(\alpha,\tau ,\mu,  \Omega )_{_{\scriptstyle  {\mathcal S}^p}}$,  вперше розглянув Л.\,В.~Тайков \cite{Taikov_1976}, \cite{Taikov_1979}.  Він знайшов точні значення деяких поперечників таких класів в просторах   $L_2$   у випадку, коли мажоранти $\Omega$ усереднених значень модулів гладкості задовольняють певні умови. Пізніше питання знаходження точних значень поперечників в просторах  $L_2$ та ${\mathcal S}^p$
 аналогічних функціональних класів, породжених різними ваговими функціями $\mu$, вивчалися
 в \cite{Aynulloyev_1984}, \cite[Гл.~4]{Pinkus_1985}, \cite{Yussef_1988, Yussef_1990, Shalaev_1991, Esmaganbetov_1999, Serdyuk_2003, Voicexivskij_2003, Vakarchuk_2004, Vakarchuk_2016} та ін.



\subsection{Означення множин $L^{\psi }{\mathcal S}^p$ і класів
$L^{\psi }(\varphi,\tau ,{v} , n)_{_{\scriptstyle  {\mathcal S}^p}}$ та $ L^{\psi }(\varphi, \tau, {v} , \Omega )_{_{\scriptstyle  {\mathcal S}^p}} $}

Нехай  $\psi=\{\psi(k)\}_{k\in {\mathbb Z}}$ -- довільна послідовність комплексних чисел. Якщо для даної функції  $f\in L$
з рядом Фур'є
 $\sum _{k\in {\mathbb Z}}\widehat {f}(k){\mathrm e}^{{\mathrm i}k
 x}
 $
 ряд
 $
 \sum _{k\in {\mathbb Z}}\psi(k)\widehat {f}(k){\mathrm e}^{{\mathrm i}k
 x}
 $
 є рядом Фур'є деякої функції   $F\in L$, то  $F$ називається  (див., наприклад, \cite[Гл.~11]{Stepanets_M2005})
 $\psi$-інтегралом функції   $f$ і позначається  $F={\mathcal J}^{\psi }(f, \cdot)$. В свою чергу функція  $f$
 називається   $\psi$-похідною функції  $F$ і позначається   $f=F^{\psi}$.
У цьому випадку коефіцієнти Фур'є функцій  $f$ та   $f^{\psi }$ пов'язані рівностями
\begin{equation} \label{Fourier_Coeff_der_Sp}
 \widehat f(k)=\psi (k)\widehat f^{\psi }(k), \quad k \in {\mathbb Z}.
  \end{equation}
Множину $\psi$-інтегралів функцій   $f$ з $L$ позначається  $L^{\psi }$. Якщо
 ${\mathfrak {N}}\subset L$, то через  $L^{\psi }{\mathfrak {N}}$ позначають множину
 $\psi$-інтегралів функцій $f\in {\mathfrak {N}}$. Зокрема,     $L^{\psi }{\mathcal S}^p$ -- множина
 $\psi$-інтегралів функцій  $f\in {\mathcal S}^p$.

У випадку, коли   $\psi (k) = ({\mathrm i}k)^{-r}$, $r=0, 1,\ldots$, будемо позначати  $L^{\psi }=:L^{r}$ і
$L^{\psi }{\mathfrak {N}}=:L^{r}{\mathfrak {N}}$.

Для довільних фіксованих  $\varphi\in \Phi$, $\tau>0$ та ${v}\in  V(\tau )$ розглянемо такі фунціональні класи:
  \begin{equation} \label{L^psi(varphi,n)}
  L^{\psi }(\varphi,\tau ,{v} , n)_{_{\scriptstyle  {\mathcal S}^p}}:=
  \Big\{f\in L^{\psi }{\mathcal S}^p:\quad
  \Omega_\varphi\Big(f^{\psi }, \tau, {v} ,\frac{\tau }n\Big)_{_{\scriptstyle  {\mathcal S}^p}} \le 1, \quad n \in
  {\mathbb{N}}\Big\},
  \end{equation}
   \begin{equation} \label{L^psi(varphi,Phi)}
   L^{\psi }(\varphi, \tau, {v} , \Omega )_{_{\scriptstyle  {\mathcal S}^p}}  :=
    \Big\{f\in L^{\psi }{\mathcal S}^p:\
    \Omega _\varphi(f^{\psi }, \tau , {v} ,  u)_{_{\scriptstyle  {\mathcal S}^p}} \le \Omega  (u),\ 0\le u\le \tau \Big\},
  \end{equation}
де $\Omega  (u)$ -- фіксована неперервна монотонно зростаюча функція змінної   $u\ge 0$ така, що  $\Omega  (0)=0$. Також позначимо   $L^{\psi }(\alpha,\tau ,{v} , n)_{_{\scriptstyle  {\mathcal S}^p}}:=L^{\psi }(\varphi,\tau ,{v} , n)_{_{\scriptstyle  {\mathcal S}^p}}$ і  $L^{\psi }(\alpha,\tau ,{v},  \Omega )_{_{\scriptstyle  {\mathcal S}^p}}:=
L^{\psi }(\varphi, \tau, {v} , \Omega )_{_{\scriptstyle  {\mathcal S}^p}}$ при   $\varphi(t)=\varphi_\alpha(t)=
2^\frac \alpha 2 (1-\cos{kh})^\frac \alpha2$.

Зазначимо, що у випадку, коли    $p=2$, $\psi (k)=k^{-r}$, $r\in \mathbb{N}$ і   ${v}(t)=t$,  класи, подібні до класів  $L^{\psi }(\alpha,\tau ,{v} , n)_{_{\scriptstyle  {\mathcal S}^p}}$ та $L^{\psi }(\alpha,\tau ,{v},  \Omega )_{_{\scriptstyle  {\mathcal S}^p}}$, вперше розглянув Л.\,В.~Тайков \cite{Taikov_1976}, \cite{Taikov_1979}.  Він знайшов точні значення  поперечників таких класів в просторах   $L_2$   у випадку, коли мажоранти $\Omega$ усереднених значень модулів гладкості задовольняють певні умови. Згодом питання знаходження точних значень поперечників в просторах  $L_2$ та ${\mathcal S}^p$
подібних функціональних класів, породжених різними ваговими функціями ${v}$, вивчалися  в \cite{Aynulloyev_1984}, \cite[Гл.~4]{Pinkus_1985}, \cite{Yussef_1988, Yussef_1990, Shalaev_1991, Esmaganbetov_1999, Serdyuk_2003, Vakarchuk_2004, Vakarchuk_2016} та ін.


\subsection{Найкращі наближення та поперечники функціональних класів}


Нехай  ${\mathscr T}_{2n+1}$, $n=0,1,\ldots$, -- множина всіх тригонометричних поліномів
${T}_{n}(x) = \sum_{|k|\le n}  c_{k}\mathrm{e}^{\mathrm{i}kx}$ порядку   $n$, де $c_{ k}$ -- довільні комплексні числа.

Для будь-якої функції  $f\in {\mathcal S}^p$ позначимо через  $E_n (f)_{_{\scriptstyle  {\mathcal S}^p}}$ її найкраще наближення
тригонометричними поліномами    ${T}_{n-1}\in {\mathscr T}_{2n-1}$ в просторі   ${\mathcal S}^p$, тобто,
 \begin{equation}\label{Best_Approx_Deff}
    E_n (f)_{_{\scriptstyle  {\mathcal S}^p}} :=
    \inf\limits_{{T}_{n-1}\in {\mathscr T}_{2n-1} }\|f-{T}_{n-1}\|_{_{\scriptstyle  {\mathcal S}^p}}
    .
 \end{equation}
Із співвідношення  (\ref{norm_Sp}) та мінімальної властивості частинних сум Фур'є випливає, що для будь-якої функції
$f \in {\mathcal S}^p$ при всіх  $n=0,1,\ldots$,
           \begin{equation} \label{Best_app_all}
           E_n^p (f)_{_{\scriptstyle  {\mathcal S}^p}} =\|f-{S}_{n-1}({f})\|_{_{\scriptstyle  {\mathcal S}^p}} ^p=\sum _{|k |\ge  n}|\widehat f(k )|^p,
           \end{equation}
де  $S_{n-1}(f)=S_{n-1}(f,\cdot)= \sum _{|k|\le n-1}\widehat{f}(k) {\mathrm{e}^{\mathrm{i}k\cdot}}$ -- частинна сума Фур'є порядку   $n-1$ функції  $f$.


Далі, нехай   ${\mathfrak N}$ -- опукла центрально-симетрична підмножина простору  ${\mathcal S}^p$ і  ${ B }$ -- одинична куля простору  ${\mathcal S}^p$. Нехай також   $F_N$ -- довільний  $N$-вимірний підпростір простору  ${\mathcal S}^p$, $N\in {\mathbb N}$, і
$\mathscr{L}({\mathcal S}^p, F_N)$ -- множина лінійних операторів з   ${\mathcal S}^p$ в  $F_N$. Через
  $\mathscr {P}({\mathcal S}^p, F_N)$ позначимо підмножину проєктивних операторів з множини  ${\mathscr{L}}({\mathcal S}^p, F_N)$,  тобто, множину операторів    $A$ лінійного проєктування на множині  $F_N$ таких, що  $Af = f$ при  $f\in F_N$.
  Величини
 \[
 b_N({\mathfrak N}, {\mathcal S}^p)=\sup\limits _{F_{N+1}}\sup\{\varepsilon>0: \varepsilon { B }\cap F_{N+1}
 \subset {\mathfrak N}\},
 \]
 \[
 d_N({\mathfrak N}, {\mathcal S}^p)=\inf\limits _{F_N}\sup \limits _{f\in {\mathfrak N}}
 \inf \limits _{u\in F_N}\|f - u \|_{_{\scriptstyle  {\mathcal S}^p}} ,
 \]
 \[
 \lambda _N({\mathfrak N},{\mathcal S}^p)=
 \inf \limits _{F_N}\inf \limits_{A\in {\mathscr {L}}({\mathcal S}^p, F_N)}\sup \limits _{f\in {\mathfrak N}}
 \|f - Af\|_{_{\scriptstyle  {\mathcal S}^p}} ,
 \]
 \[
 \pi _N({\mathfrak N}, {\mathcal S}^p)=\inf \limits_{F_N}\inf \limits _{A\in {\mathscr {P}}({\mathcal S}^p,F_N)}
 \sup \limits _{f\in {\mathfrak N}}\|f - Af\|_{_{\scriptstyle  {\mathcal S}^p}} ,
 \]
називаються бернштейнівським, колмогоровським, лінійним та проєктивним $N$-поперечниками множини
${\mathfrak N}$ в просторі  ${\mathcal S}^p$, відповідно.




\subsection{Прямі апроксимативні теореми в просторах ${\mathcal S}^p$}\label{DATHs}

Наведемо відомі на даний час  нерівності типу Джексона в термінах найкращих наближень функцій та усереднених значень їх узагальнених модулів гладкості в просторах  ${\mathcal S}^p$.

 \begin{theorem}{\rm (\cite{Serdyuk_2003, Abdullayev_Serdyuk_Shidlich_2021})}
       \label{Th.3.1.}
      Нехай $f\in L^{\psi }{\mathcal S}^p$, $1\le p<\infty$,  $\varphi\in \Phi$, $\tau>0$,  ${v}\in  V(\tau )$
      і $\{\psi (k)\}_{k\in {\mathbb Z}}$ -- довільна послідовність комплексних чисел таких, що   $|\psi(k)|\le K<\infty$.
      Тоді для довільного    $n\in {{\mathbb N}}$ справджується нерівність
 \begin{equation}\label{Jackson_Type_Ineq}
    E_n(f)_{_{\scriptstyle  {\mathcal S}^p}} \le
    \bigg(\frac {{v} (\tau ) - {v} (0)}
    {I_{n,\varphi,p}(\tau ,{v} )}\bigg)^{1/p} \nu (n)\,
    \Omega _\varphi\Big(f^{\psi}, \tau,{v} , \frac{\tau }n\Big)_{_{\scriptstyle  {\mathcal S}^p}},
 \end{equation}
де   $\nu(n):=\nu(n,\psi)=\sup_{|k|\ge n} |\psi(n)|$,
 \begin{equation}\label{I_n,varphi,p}
      I_{n,\varphi,p}(\tau ,{v} ):= 
      \mathop{\inf\limits _{k\ge n}}\limits_{k \in {\mathbb N}} \int\limits _0^{\tau }\varphi^p\Big(\frac {k t}n\Big)
      {\mathrm d} {v}   (t).
 \end{equation}
Якщо при цьому функція  $\varphi$ є неспадною на проміжку $[0,\tau]$,  величина
 $\nu(n)=\max\{|\psi(n)|,|\psi(-n)|\}$, і виконується умова
 \begin{equation}\label{I_n,varphi,p_Equiv_n}
      I_{n,\varphi,p}(\tau ,{v} )=\int\limits _0^{\tau }\varphi^p(t) {\mathrm d} {v}   (t),
 \end{equation}
то нерівність  $(\ref{Jackson_Type_Ineq})$ не може бути покращена і тому
 \begin{equation}
       \label{Jackson_Type_Exact}
       \mathop {\sup\limits _{f\in L^{\psi}{\mathcal S}^p}}\limits _{f\not ={\rm const }}
       \frac {E_n(f)_{_{\scriptstyle  {\mathcal S}^p}} }
       {\Omega_\varphi (f^{\psi }, \tau, {v}, \frac{\tau }{n} )_{_{\scriptstyle  {\mathcal S}^p}} }=
       \bigg(\frac {{v}(\tau ) - {v} (0)}{\int_0^{\tau }\varphi^p(t) {\mathrm d} {v}   (t)}\bigg)^{1/p}\nu (n).
 \end{equation}

\end{theorem}

Внаслідок \eqref{general_modulus_INEQ} із теореми \ref{Th.3.1.}, отримуємо  твердження:

\begin{corollary}{\rm (\cite{Serdyuk_2003, Abdullayev_Serdyuk_Shidlich_2021})} \label{Cor.3.1}
Нехай $f\in L^{\psi }{\mathcal S}^p$, $1\le p<\infty$, $\varphi\in \Phi$, $\tau>0$, ${v}\in  V(\tau )$ і $\{\psi (k)\}_{k\in {\mathbb Z}}$ -- послідовність комплексних чисел таких, що   $|\psi(k)|\le K<\infty$.  Тоді при будь-якому
$n\in {{\mathbb N}}$
 \begin{equation}\label{Jackson_type_OLD_PSI}
       E_n(f)_{_{\scriptstyle  {\mathcal S}^p}} \le \bigg(\frac {{v} (\tau ) - {v} (0)}
    {I_{n,\varphi,p}(\tau ,{v} )}\bigg)^{1/p}\,\nu(n)\, \omega_\varphi\Big(f,\frac{\tau}{n}\Big)_{_{\scriptstyle  {\mathcal S}^p}},
   \end{equation}
де  $\nu(n)=\sup_{|k|\ge n} |\psi(n)|$, а величина  $I_{n,\varphi,p}(\tau ,{v} )$ визначається рівністю $(\ref{I_n,varphi,p})$.
\end{corollary}

Функція $ \varphi_\alpha(t)=2^\frac \alpha 2 (1-\cos{t})^\frac \alpha2$, $\alpha>0$, є неспадною на відрізку   $[0,\pi]$. Тому з теореми \ref{Th.3.1.} випливає


\begin{corollary}{\rm (\cite{Serdyuk_2003, Abdullayev_Serdyuk_Shidlich_2021})} \label{Cor.3.2}
        Нехай $f\in L^{\psi }{\mathcal S}^p$, $1\le p<\infty$,   $\tau>0$, ${v}\in  V(\tau )$ і
        $\{\psi (k)\}_{k\in {\mathbb Z}}$ -- послідовність комплексних чисел таких, що
        $|\psi(k)|\le K<\infty$. Тоді для довільних чисел  $\alpha>0$ та $n\in {{\mathbb N}}$
     $$
    E_n(f)_{_{\scriptstyle  {\mathcal S}^p}} \le
    \bigg(\frac {{v} (\tau ) - {v} (0)}
    {I_{n,\alpha,p}(\tau ,{v} )}\bigg)^{1/p}\,\nu(n)\,
    \Omega _\alpha\Big(f^{\psi}, \tau,{v} , \frac{\tau }n\Big)_{_{\scriptstyle  {\mathcal S}^p}},
    \eqno (\ref{Jackson_Type_Ineq}')
    $$
де  $\nu(n)=\sup_{|k|\ge n} |\psi(n)|$, величина  $I_{n,\alpha,p}(\tau ,{v} )$ визначається рівністю
  $(\ref{I_n,varphi,p})$ з  $ \varphi(t)=\varphi_\alpha(t)=
2^\frac \alpha 2 (1-\cos{t})^\frac \alpha2
$.

Якщо при цьому  $\nu(n)=\max\{|\psi(n)|,|\psi(-n)|\}$  і
   $$
   I_{n,\alpha,p}(\tau ,{v} )=
   2^\frac {\alpha p} 2 \int\limits _0^{\tau }(1-\cos t)^\frac {\alpha p}2 {\mathrm d} {v}   (t),
   \eqno(\ref{I_n,varphi,p_Equiv_n}')
   $$
то при   $\tau \in (0, \pi] $ нерівність  $(\ref{Jackson_Type_Ineq}')$ не може бути покращена і тому
  $$
       \mathop {\sup\limits _{f\in L^{\psi}{\mathcal S}^p}}\limits _{f\not ={\rm const }}
       \frac {E_n(f)_{_{\scriptstyle  {\mathcal S}^p}} }
       {\Omega_\alpha (f^{\psi }, \tau, {v}, \frac{\tau }{n} )_{_{\scriptstyle  {\mathcal S}^p}} }=
       \bigg(\frac {{v}(\tau ) - {v} (0)}{2^\frac {\alpha p} 2\int_0^{\tau }(1-\cos t)^{\frac {\alpha p}2}
       {\mathrm d} {v}   (t)}\bigg)^{1/p}\,\nu(n)\,
  $$
  $$
       =
       \bigg(\frac {{v}(\tau ) - {v} (0)}{2^{\alpha p}\int_0^{\tau }\sin^{\alpha p} \frac t2
       {\mathrm d} {v}   (t)}\bigg)^{1/p}\,\nu(n)\,
       .
       \eqno(\ref{Jackson_Type_Exact}')
  $$

\end{corollary}


Наведемо деякі наслідки з цього твердження для вагових функцій ${v} _1(t) = 1 - \cos t$ та ${v} _2(t)=t$.

\begin{corollary}{\rm (\cite{Serdyuk_2003, Abdullayev_Serdyuk_Shidlich_2021})} \label{Cor.3.3}
           Нехай $f\in L^{\psi }{\mathcal S}^p$, $1\le p<\infty$,  і   $\{\psi (k)\}_{k\in {\mathbb Z}}$ -- послідовність
           комплексних чисел таких, що   $|\psi(k)|\le K<\infty$.  Тоді для довільних    $\alpha>0$ та $n\in {{\mathbb N}}$
 \begin{equation}\label{A6.98}
        E_n(f)_{_{\scriptstyle  {\mathcal S}^p}} \le
               \bigg(\frac {2}
               {I_{n,\alpha,p}
               (\pi , {v} _1)}\bigg)^{1/p}
               \Omega_\alpha\Big(f^{\psi }, \pi, {v} _1, \frac {\pi }n\Big)_{_{\scriptstyle  {\mathcal S}^p}}
               \,\nu(n)\,
               ,
  \end{equation}
де  $\nu(n) =\sup_{|k|\ge n} |\psi(n)|$,
 \begin{equation}\label{A6.99} 
       I_{n,\alpha,p}(\pi ,{v}_1)
         =2^{\frac {\alpha p} 2}
       \mathop{\inf\limits _{k\ge n}}\limits_{k \in {\mathbb N}} \int\limits _0^{\pi}\Big(1 - \cos
      \frac {k t}n\Big)^{\frac {\alpha p }2} \sin t \, {\mathrm d}t
  \end{equation}
Якщо при цьому    $\nu(n)=\max\{|\psi(n)|,|\psi(-n)|\}$  і
$\frac {\alpha p }2\in {{\mathbb N}}$, то нерівність
$(\ref{A6.98})$ на множині  $L^{\psi }{\mathcal S}^p$ не може бути покращена і
  \begin{equation}\label{A6.100} 
  \mathop {\sup\limits _{f\in L^{\psi }{\mathcal S}^p}}\limits _{f\not = {\rm const}}
  \frac {E_n(f)_{_{\scriptstyle  {\mathcal S}^p}} }{\Omega _\alpha (f^{\psi }, \pi, {v}_1,
\frac {\pi }n)_{_{\scriptstyle  {\mathcal S}^p}} } = \frac {(\frac {\alpha p }2 + 1)^{1/p}}{2^\alpha}
 \,\nu(n)
 .
  \end{equation}

\end{corollary}

У випадку, коли  $p=2$, ${v} _1(t) = 1 - \cos t$, $\tau=\pi$ і $\varphi(t)=2^\frac 12 (1-\cos{t})^\frac 12$, тобто, коли
   $\omega_\varphi$ -- звичайний модуль неперервності (модуль гладкості порядку 1), нерівність вигляду
   (\ref{Jackson_type_OLD_PSI}) наведено в  \cite[Гл.~6]{Stepanets_M2005}. При цьому
 $
 \Big(\frac {{v} (\tau ) - {v} (0)}
    {I_{n,\varphi,p}(\tau ,{v} )}\Big)^{1/p}=2^{-1/2}.
 $

\begin{remark}
 У випадку, коли  $p=2$ і $\psi (k) = ({\mathrm i}k)^{-r}$, $r=0, 1,\ldots$,  рівність $(\ref{A6.100})$ може бути записана у вигляді
  $$
         \mathop {\sup\limits _{f\in L^r {\mathcal S}^2}}\limits _{f\not ={\rm const}}
         \frac{E_n(f)_{_{\scriptstyle  {\mathcal S}^2}} }
         {\Omega _\alpha (f^{(r)}, \pi, {v} _1, \frac {\pi }n)_{_{\scriptstyle  {\mathcal S}^2}}}
         =\frac{\sqrt {\alpha+1}}{2^\alpha}n^{-r}, \quad \alpha>0, \ n\in {\mathbb N}.
         \eqno (\ref{A6.100}')
 $$
 При $\alpha=1$ це співвідношення випливає з результату М.\,І.~Черниха  \cite{Chernykh_1967_MZ}.
При  довільних  $\alpha=k\in {{\mathbb N}}$ і  $n\in {\mathbb N}$ точні значення величин в лівій частині
$(\ref{A6.100}')$ отримано Х.~Юссефом \cite{Yussef_1988} в дещо іншій формі.
 \end{remark}

\begin{corollary} {\rm (\cite{Serdyuk_2003, Abdullayev_Serdyuk_Shidlich_2021, Voicexivskij_UMZh2003})}  \label{Cor.3.4}  Нехай   $0<\tau \le \frac {3\pi }4$, ${v} _2(t)=t$,
  числа  $1\le p<\infty$ та  $\alpha>0$ такі, що  $\alpha p\ge 1$. Нехай також  $n\in {{\mathbb N}}$  і
  $\{\psi (k)\}_{k\in {\mathbb Z}}$ -- послідовність комплексних чисел таких, що
   $\nu(n)=\sup_{|k|\ge n} |\psi(n)| =\max\{|\psi(n)|,|\psi(-n)|\}$. Тоді
    \begin{equation}\label{A6.103}
    \mathop{\mathop {\sup}\limits _{f\in L^{\psi }{\mathcal S}^p}}\limits_{f\not ={\rm const }}
           \frac {E_n(f)_{_{\scriptstyle  {\mathcal S}^p}} }
                 {\Omega _\alpha(f^{\psi },\tau , {v}_2, \frac {\tau }n)_{_{\scriptstyle  {\mathcal S}^p}} }
                 = \bigg(\frac {\tau}{2^{\alpha p}\int _0^{\tau }\sin ^{\alpha p }\frac t{2}{\mathrm d}t}\bigg)^{1/p}
                 \nu(n)
                 .
  \end{equation}

\end{corollary}

Зазначимо, що у випадку, коли $p=2$, $\psi (k)=({\mathrm i}k)^{-r},$ $r\ge 0$ і
 $k=1$ або $r\ge 1/2$ і $k \in {{\mathbb N}}$, рівність  (\ref{A6.103}) випливає з результатів Л.\,В.~Тайкова
 \cite{Taikov_1976, Taikov_1979}.


\subsection{Поперечники класів $L^{\psi }(\varphi, {v}, \tau, n)_{_{\scriptstyle  {\mathcal S}^p}} $}

В цьому підрозділі наведемо твердження про точні значення колмогоровських, бернештейнівських, лінійних та проєктивних поперечників  класів $L^{\psi}(\varphi, {v}, \tau, n)_{_{\scriptstyle  {\mathcal S}^p}} $, означених формулою (\ref{L^psi(varphi,n)}), у випадку, коли послідовності
$\psi (k)$ задовольняють деякі природні умови. Позначимо через
$\Psi$ множину всіх послідовностей $\{\psi (k)\}_{k\in {\mathbb Z}}$  комплексних чисел таких, що
$|\psi (k)|=|\psi (-k)| \ge |\psi (k+1)|$ при $k\in {\mathbb N}$.

\begin{theorem} {\rm (\cite{Serdyuk_2003, Abdullayev_Serdyuk_Shidlich_2021})} \label{Th.3.2}
       Нехай $1\le p<\infty,$ $\psi \in \Psi$, $\tau>0$, функція
       $\varphi\in \Phi$ є неспадною на відрізку $[0,\tau]$ і  ${v} \in V(\tau )$.
       Тоді для довільних  $n\in {\mathbb N}$ і  $N\in \{2n-1,  2n\}$  справджуються нерівності
\[
   \bigg(\frac{{v} (\tau ) - {v} (0)}
   {\int_0^{\tau }\varphi^p(t) {\mathrm d} {v}   (t)}\bigg)^{1/p}|\psi (n)|
   \le P_{N} (L^{\psi }(\varphi, \tau, {v}, n)_{_{\scriptstyle  {\mathcal S}^p}} , {\mathcal S}^p )
\]
\begin{equation}\label{A6.105}
    \le \bigg(\frac {{v} (\tau ) - {v}(0)}
   {I_{n,\varphi,p}(\tau ,{v} )}\bigg)^{1/p}   |\psi (n)|,
  \end{equation}
де величина $I_{n,\varphi,p}(\tau ,{v} )$ означається рівністю  $(\ref{I_n,varphi,p})$, і $P_N$
-- будь-який із поперечників  $b_N$, $d_N$, $\lambda _N$ чи $ \pi _N$. Якщо при цьому виконується умова
  $(\ref{I_n,varphi,p_Equiv_n})$, то
    \begin{equation}\label{A6.106} 
          P_{N} (L^{\psi }(\varphi, \tau, {v}, n)_{_{\scriptstyle  {\mathcal S}^p}} , {\mathcal S}^p )=
          \bigg(\frac{{v} (\tau ) - {v} (0)}
          {\int_0^{\tau }\varphi^p(t) {\mathrm d} {v}   (t)}\bigg)^{1/p}|\psi (n)|.
      \end{equation}
\end{theorem}

З теореми \ref{Th.3.2} у випадку, коли $ \varphi(t)=
2^\frac \alpha 2 (1-\cos{t})^\frac \alpha2$, отримуємо таке твердження.


\begin{corollary} {\rm (\cite{Serdyuk_2003, Abdullayev_Serdyuk_Shidlich_2021})} \label{Cor.3.5}
                 Нехай  $1\le p<\infty$, $\psi \in \Psi,$ $\tau \in (0, \pi],$ $\alpha\in {\mathbb N}$ і
                 ${v}  \in V(\tau )$. Тоді для довільних  $n\in {\mathbb N}$ і  $N\in \{2n-1,  2n\}$
                 справджуються нерівності
    \[ 
                  \bigg(\frac {{v}(\tau ) - {v} (0)}{2^{\alpha p}\int_0^{\tau }\sin^{\alpha p} \frac t2 {\mathrm d} {v}   (t)}\bigg)^{1/p}|\psi (n)|\le P_N(L^{\psi }(\alpha, \tau, {v}, n)_{_{\scriptstyle  {\mathcal S}^p}} , {\mathcal S}^p)
      \le
                   \bigg(\frac {{v} (\tau ) - {v} (0)}
                  {I_{n,\alpha,p}(\tau ,{v} )}\bigg)^{1/p}|\psi (n)|,
    \]
де величина  $I_{n,\alpha,p}(\tau ,{v} )$ означається рівністю  $(\ref{I_n,varphi,p})$ з
$ \varphi(t)=
2^\frac \alpha 2 (1-\cos{kh})^\frac \alpha2$, і $P_N$ -- будь-який із поперечників  $b_N$, $d_N$, $\lambda _N$ чи $ \pi _N$.
Якщо при цьому виконується умова  $(\ref{I_n,varphi,p_Equiv_n}')$, то
   \[
               P_{N}(L^{\psi }(\alpha, \tau, {v}, n)_{_{\scriptstyle  {\mathcal S}^p}} , {\mathcal S}^p)=
               \bigg(\frac {{v}(\tau ) - {v} (0)}{2^{\alpha p}\int_0^{\tau }\sin^{\alpha p}
               \frac t2 {\mathrm d} {v}   (t)}\bigg)^{1/p}\!\!|\psi (n)|.
    \]

\end{corollary}

Для вагових функцій   ${v} _1(t)= 1 - \cos t$ та ${v} _2(t)=t$ з наслідку   \ref{Cor.3.5} випливають такі твердження.

\begin{corollary} \label{Cor.3.6}
                 Нехай  $1\le p<\infty$, $\psi \in \Psi$, $\alpha\in {\mathbb N}$ і  ${v}_1(t)=1-\cos t$.
                 Тоді для довільних  $n\in {\mathbb N}$ та  $N\in \{2n-1,  2n\}$
\[%
              \frac {(\frac {\alpha p }2+1)^{1/p}}{2^\alpha }|\psi(n)|\le
              P _N(L^{\psi }(\alpha, \pi, {v} _1, n)_{_{\scriptstyle  {\mathcal S}^p}} , {\mathcal S}^p)\le
              \bigg(\frac {2}{I_{n,\alpha,p}
               (\pi , {v} _1)}\bigg)^{1/p}
               |\psi(n)|,
\]%
де $I_{n,\alpha,p}(\pi ,{v}_1)$  -- величина вигляду  $(\ref{A6.99})$, і $P_N$ --
 будь-який із поперечників  $b_N$, $d_N$, $\lambda _N$ чи $ \pi _N$. Якщо при цьому число
 $\frac {\alpha p}2\in {{\mathbb N}},$ то
      \[
      P_{N}(L^{\psi }(\alpha, \pi, {v} _1, n)_{_{\scriptstyle  {\mathcal S}^p}} ,{\mathcal S}^p)=
       \frac {(\frac {\alpha p }2+1)^{1/p}}{2^\alpha }|\psi(n)|.
       \]
\end{corollary}

\begin{corollary} {\rm (\cite{Voicexivskij_2003})}   \label{Cor.3.7}
                  Нехай  $\psi \in \Psi ,$ $0<\tau \le \frac {3\pi }4$, ${v} _2(t)=t$, числа
                   $\alpha>0$ та $1\le p<\infty$  такі, що
                   $\alpha p\ge 1$. Тоді для довільних   $n\in {\mathbb N}$ та  $N\in \{2n-1,  2n\}$
     \[ 
     P_{N}(L^{\psi }(\alpha, \tau, {v} _2,  n)_{_{\scriptstyle  {\mathcal S}^p}},{\mathcal S}^p)=
     \bigg(\frac {\tau}{2^{\alpha p}\int _0^{\tau }\sin ^{\alpha p }\frac t{2}{\mathrm d}t}\bigg)^{1/p}
     |\psi (n)|,
     \] 
 де $P_N$ --  будь-який із поперечників  $b_N$, $d_N$, $\lambda _N$ чи $ \pi _N$.

\end{corollary}


\subsection{Поперечники класів   $L^{\psi }(\varphi, {v}, \tau, \Omega )_{_{\scriptstyle  {\mathcal S}^p}} $}

Наведемо результати про точні  значення поперечників класів
 $L^{\psi }(\varphi, {v}, \tau,\Omega )_{_{\scriptstyle  {\mathcal S}^p}} $ вигляду (\ref{L^psi(varphi,Phi)}) при деяких обмеженнях на мажоранти   $\Omega$.

\begin{theorem}
       \label{Th.3.3}
      Нехай $1\le p<\infty$, $\psi \in \Psi$, функція
       $\varphi\in \Phi$ є неспадною на деякому відрізку
       $[0,a]$, $a>0$, і   $\varphi(a)=\sup\{\varphi(t):\, t\in {\mathbb R}\}$.
       Нехай також  $\tau\in (0,a]$,  ${v} \in V(\tau )$ і при всіх $\xi>0$ та $0<u\le a $
       функція   $\Omega$ задовольняє умову
     \begin{equation}\label{A6.113} 
     \Omega  \Big(\frac u{\xi }\Big)\bigg(\int\limits _0^{\xi \tau}\varphi_{*}^p(t) {\mathrm d} {v}
     \Big(\frac t{\xi }\Big)\bigg)^{1/p}
     \le \Omega  (u)\bigg(\int \limits _0^{\tau }\varphi^p(t)   {\mathrm d} {v}  (t)\bigg)^{1/p},
      \end{equation}
      де
     \begin{equation}\label{A6.114} 
      \varphi_{*}(t):=\left \{\begin{matrix} \varphi(t),\quad \hfill & 0\le t\le a, \\
      \varphi(a),\quad \hfill &  t\ge a.\end{matrix}\right.
      \end{equation}
Тоді для довільних  $n\in {\mathbb N}$ та  $N\in \{2n-1,  2n\}$
     \[
            \bigg(\frac{{v} (\tau ) - {v} (0)}{\int_0^{\tau }\varphi^p(t) {\mathrm d} {v}   (t)}\bigg)^{1/p}
            |\psi (n)|\,\,\Omega  \Big(\frac {\tau }n\Big)\le
            P_N(L^{\psi }(\varphi, \tau, {v}, \Omega)_{_{\scriptstyle  {\mathcal S}^p}} , {\mathcal S}^p)
     \]
     \begin{equation}\label{A6.115} %
            \le
            \bigg(\frac {{v} (\tau ) - {v}(0)} {I_{n,\varphi,p}(\tau ,{v} )}\bigg)^{1/p}  |\psi (n)|
            \,\,\Omega  \Big(\frac {\tau }n\Big),
      \end{equation}
де величина  $I_{n,\varphi,p}(\tau ,{v} )$ означена рівністю  $(\ref{I_n,varphi,p})$, і
$P_N$ --  будь-який із поперечників  $b_N$, $d_N$, $\lambda _N$ чи $ \pi _N$. Якщо при цьому виконується умова
 $(\ref{I_n,varphi,p_Equiv_n})$, то
     \begin{equation}\label{A6.116} 
    P_{N}(L^{\psi } (\varphi, \tau,  {v}, \Omega  )_{_{\scriptstyle  {\mathcal S}^p}} , {\mathcal S}^p)=
     \bigg(\frac{{v} (\tau ) - {v} (0)}{\int_0^{\tau }\varphi^p(t) {\mathrm d} {v}   (t)}\bigg)^{1/p}
            |\psi (n)|\,\,\Omega  \Big(\frac {\tau }n\Big).
      \end{equation}

\end{theorem}

У випадку  $ \varphi(t)=\varphi_\alpha(t)=
2^\frac \alpha 2 (1-\cos{t})^\frac \alpha2
$, справджується таке твердження:


\begin{corollary} \label{Cor.3.8}
                  Нехай $1\le p<\infty$, $\psi \in \Psi$, $\tau \in (0, \pi],$ $\alpha>0$ і  ${v} \in V(\tau )$.
                  Нехай також при всіх  $\xi>0$ та $0<u\le \pi $, функція  $\Omega$ задовольняє умову
                  $$
                  \Omega  \Big(\frac u{\xi }\Big)
                  \bigg(\int\limits _0^{\xi \tau}(1 - \cos t)_{*}^{\frac {\alpha p }2}{\mathrm d} {v}
                  \Big(\frac {t}{\xi}\Big)\bigg)^{1/p}
                  \le
                  \Omega  (u)\bigg(\int \limits _0^{\tau }(1-\cos t)^{\frac {\alpha p }2}{\mathrm d} {v}  (t)\bigg)^{1/p},
                  \eqno (\ref{A6.113}')
                  $$
де
                  $$
                  (1-\cos t)_{*}:=\left \{\begin{matrix} 1-\cos t,\quad \hfill & 0\le t\le \pi, \\
                  2,\quad \hfill &  t\ge \pi.\end{matrix}\right.
                  \eqno (\ref{A6.114}')
                  $$
Тоді для довільних $n\in {\mathbb N}$ та $N\in \{2n-1,  2n\}$ справджуються нерівності
                  $$
                   \bigg(\frac {{v}(\tau ) - {v} (0)}
                   {2^{\alpha p}\int_0^{\tau }\sin^{\alpha p} \frac t2 {\mathrm d} {v} (t)}\bigg)^{1/p}
                   |\psi (n)| \Omega  \Big(\frac {\tau }n\Big)\le
                   P_N(L^{\psi}(\alpha, \tau, {v}, \Omega)_{_{\scriptstyle  {\mathcal S}^p}} , {\mathcal S}^p)
                   $$
                   $$
                   \le \bigg(\frac {{v} (\tau ) - {v} (0)}
                   {I_{n,\alpha,p}(\tau ,{v} )}\bigg)^{1/p}
                   |\psi (n)| \Omega  \Big(\frac {\tau }n\Big),
                   $$
де величина  $I_{n,\alpha,p}(\tau ,{v} )$ означена рівністю  $(\ref{I_n,varphi,p})$ з
$ \varphi(t)=2^\frac \alpha 2 (1-\cos{kh})^\frac \alpha2$, і $P_N$ --  будь-який із поперечників  $b_N$, $d_N$, $\lambda _N$ чи $ \pi _N$.
Якщо при цьому виконується умова  $(\ref{I_n,varphi,p_Equiv_n}')$, то
                 $$
                 P_{N}(L^{\psi }(\alpha, \tau, {v},  \Omega  )_{_{\scriptstyle  {\mathcal S}^p}} , {\mathcal S}^p)=
                 \bigg(\frac {{v}(\tau ) - {v} (0)}
                 {2^{\alpha p}\int_0^{\tau }\sin^{\alpha p} \frac t2 {\mathrm d} {v} (t)}\bigg)^{1/p}
                 |\psi (n)| \Omega  \Big(\frac {\tau }n\Big).
                 $$

\end{corollary}

Зазначимо, що для   конкретних функцій
${v}  \in V(\tau )$ при певних  обмеженнях на ті чи інші параметри задачі питання існування функцій  $\Omega$, які задовольняють умови вигляду  (\ref{A6.113}) та $(\ref{A6.113}')$, досліджувалась в    \cite{Taikov_1976, Taikov_1979, Aynulloyev_1984, Yussef_1990} та ін.

Для вагових функцій ${v} _1(t)= 1 - \cos t$ та ${v} _2(t)=t$ з наслідку  \ref{Cor.3.8} випливають такі  твердження.

\begin{corollary} \label{Cor.3.9}
                  Нехай $1\le p<\infty$, $\psi \in \Psi$, $\alpha>0$, ${v} _1(t)=1-\cos t$ і при всіх
                  $\xi >0$ та $0<u\le \pi $ функція  $\Omega  $ задовольняє умову
     \begin{equation}\label{A6.121} 
           \Omega  \Big(\frac u{\xi}\Big)
           \bigg (\frac 1{\xi }\int \limits _0^{\pi \xi}(1 - \cos t)_{*}^{\frac {\alpha p }2}\sin \frac t{\xi }dt\bigg)^{1/p}
          \!\!\! \le \Omega  (u)\bigg
           (\int \limits _0^{\pi}(1-\cos t)^{\frac {\alpha p }2}\sin tdt \bigg )^{1/p}\!\!,
     \end{equation} 
де функція  $(1-\cos t)_{*}$ задається співвідношенням  $(\ref{A6.114}')$. Тоді для довільних
$n\in {\mathbb N}$ та  $N\in \{2n-1,  2n\}$
 \[%
              \frac {(\frac {\alpha p }2+1)^{1/p}}{2^\alpha }|\psi(n)|
              \Omega \Big(\frac {\tau }n\Big)
              \le P_N(L^{\psi }(\alpha,  \pi, {v} _1, \Omega)_{_{\scriptstyle  {\mathcal S}^p}} , {\mathcal S}^p)\le
              \frac {2^{1/p} |\psi(n)|}
              {I^{1/p}_{n,\alpha,p}
               (\pi , {v} _1)}\Omega \Big(\frac {\tau }n\Big),
  \]
де $I_{n,\alpha,p}(\pi ,{v}_1)$ -- величина вигляду  $(\ref{A6.99})$, і $P_N$
-- будь-який із поперечників  $b_N$, $d_N$, $\lambda _N$ чи $ \pi _N$. Якщо при цьому
$\frac {\alpha p }2\in {{\mathbb N}},$, то
      \[ 
            P_{N}(L^{\psi }(\alpha,  \pi, {v}_1, \Omega )_{_{\scriptstyle  {\mathcal S}^p}} ,{\mathcal S}^p)=
            \frac {(\frac {\alpha p }2+1)^{1/p}}{2^\alpha }|\psi(n)|
            \Omega \Big(\frac {\tau }n\Big).
      \]  

\end{corollary}

У випадку, коли    $p=2$, $\psi(k)=({\mathrm i}k)^{-r},$ $r\in {{\mathbb N}},$ і $\alpha=1$, наслідок
 \ref{Cor.3.9} отримано Н.~Айнуллоєвим  \cite{Aynulloyev_1984}. В \cite{Aynulloyev_1984} було також доведено існування функцій $\Omega$, які задовольняють  (\ref{A6.121}) при наведених вище обмеженнях на параметри  $p$ та $\alpha$.

\begin{corollary} \label{Cor.3.10}
                Нехай  $1\le p<\infty$, $\psi \in \Psi$, $0<\tau \le \frac {3\pi }4,$ ${v}_2 = t$
                і при всіх  $\xi >0$  та $0<u\le \pi $ функція  $\Omega  $ задовольняє умову
      \begin{equation}\label{A6.123} 
              \Omega  \Big(\frac {u}{\xi }\Big)
              \bigg (\frac 1{\xi}\int \limits _0^{\xi \tau }(1 - \cos t)_{*}^{\frac {\alpha p }2}dt\bigg )^{1/p}
              \le \Omega  (u)
              \bigg (\int \limits _0^{\tau }(1- \cos t)^{\frac {\alpha p }2}dt \bigg )^{1/p}.
      \end{equation} 
Тоді для довільних $n\in {\mathbb N}$ та  $N\in \{2n-1,  2n\}$
 \[
            P_{N}(L^{\psi }(\alpha, \tau, {v}_2, \Omega )_{_{\scriptstyle  {\mathcal S}^p}} ,{\mathcal S}^p)=
            \bigg(\frac {\tau}{2^{\alpha p}\int _0^{\tau }\sin ^{\alpha p }\frac t{2}{\mathrm d}t}\bigg)^{1/p}
            |\psi (n)|\Omega  \Big(\frac {\tau }n\Big),
 \]
де $P_N$ -- будь-який із поперечників  $b_N$, $d_N$, $\lambda _N$ чи $ \pi _N$.

\end{corollary}

Зазначимо, що для $\alpha \in {\mathbb N}$ наслідки \ref{Cor.3.2} та \ref{Cor.3.3} (при $\psi \in \Psi$), а також  \ref{Cor.3.5}, \ref{Cor.3.8} та \ref{Cor.3.9} встановлено раніше А.\,С.~Сердюком \cite{Serdyuk_2003}.

У випадку, коли   $p=2,$ $\psi (k)=({\mathrm i}k)^{-r}$ і  $r\ge 0,$ $\alpha=1$ або
$r\ge 1/2,$ $\alpha\in {{\mathbb N}},$ наслідок  \ref{Cor.3.10} випливає з результатів робіт  \cite{Taikov_1976}, \cite{Taikov_1979} (див. також  \cite[Гл.~4]{Pinkus_1985}), де також доведено  існування функцій $\Omega$, які задовольняють (\ref{A6.123}) з відповідними обмеженнями на параметри
 $p,$ $\alpha$ та $r$.

Питання отримання нерівностей типу Джексона в просторах ${\mathcal S}^p$, а також знаходження точних значень поперечників класів, породжених обмеженнями на усереднені значення модулів гладкості, аналогічних до  (\ref{Mean_Value_Gen_Modulus}), розглядалися також в роботах  \cite{Voicexivskij_2002, Vakarchuk_2004, Vakarchuk_Shchitov_2006, Abdullayev_Chaichenko_Shidlich_2021_RMJ} та ін.




\enddocument